\documentclass[a4paper,11pt]{article}
\usepackage{amsmath,mathrsfs,amsthm,amsfonts,amssymb,wasysym,vmargin}
\usepackage{paralist}
\usepackage{hyperref}
\usepackage{booktabs}
\usepackage{times}
\usepackage{graphicx}
\usepackage[numbers,sort&compress]{natbib}
\usepackage[normalem]{ulem}
\usepackage{tikz}
\usetikzlibrary{shapes,backgrounds}

\graphicspath{{_figures/}}

\setpapersize{A4} 
\hypersetup{
    unicode=false,          
    pdftoolbar=true,        
    pdfmenubar=true,        
    pdffitwindow=true,      
    pdftitle={My title},    
    pdfauthor={Author},     
    pdfsubject={Subject},   
    pdfnewwindow=true,      
    pdfkeywords={keywords}, 
    colorlinks=true,        
    linkcolor=blue,         
    citecolor=blue,         
    filecolor=blue,         
    urlcolor=blue           
}

\newtheorem{lemma}{Lemma}
\newtheorem{theorem}{Theorem}
\newtheorem{cor}{Corollary}
\newtheorem{proposition}{Proposition}

\newcommand{\di}{\diamond}

\newcommand{\cI}{\mathcal I}

\newcommand{\cA}{\mathcal A}

\newcommand{\cG}{\mathcal G}

\newcommand{\cS}{\mathscr S}
\newcommand{\cU}{\mathcal U}

\newcommand{\bX}{\mathbf X}
\newcommand{\bY}{\mathbf Y}

\newcommand{\sC}{\mathscr C}

\newcommand{\bin}{\textup{Bin}}
\newcommand{\set}[1]{\left( #1 \right)}
\newcommand{\pran}[1]{\left( #1 \right)}
\newcommand{\event}[1]{\left[\, #1 \,\right]}
\newcommand{\Z}{{\mathbb Z}}
\newcommand{\N}{{\mathbb N}}
\newcommand{\R}{{\mathbb R}}
\newcommand{\Ec}[1]{{\mathbf E}[#1]}
\newcommand{\E}[1]{{\mathbf E}\left[#1\right]}
\newcommand{\CExp}[2]{\mathbf{E}\event{\left. #1 \; \right| \; #2}}
\newcommand{\e}{{\mathbf E}}
\newcommand{\p}[1]{{\mathbf P}\left(#1\right)}
\newcommand{\pc}[1]{{\mathbf P}(#1)}
\newcommand{\Cprob}[2]{\mathbf{P}\set{\left. #1 \; \right| \; #2}}

\newcommand{\EXP}{\e}

\newcommand{\flr}[1]{\ensuremath{\lfloor #1 \rfloor}}
\newcommand{\cel}[1]{\ensuremath{\lceil #1 \rceil}}

\newcommand{\noi}{\noindent}

\usepackage[refpage,noprefix,intoc]{nomencl}

\makenomenclature                                          
\date{\today}
\title{Almost optimal sparsification of random geometric graphs}
\author{
Nicolas Broutin\thanks{Projet RAP, Inria Paris--Rocquencourt} \and 
Luc Devroye\thanks{McGill University} \and 
G\'abor Lugosi\thanks{ICREA and Pompeu Fabra University. 
GL acknowledges support by the Spanish Ministry of Science and Technology grant MTM2012-37195.}
}

\begin{document}


\maketitle

\begin{abstract}
A random geometric irrigation graph $\Gamma_n(r_n,\xi)$ has $n$ vertices 
identified by $n$ independent uniformly distributed points $X_1,\ldots,X_n$
in the unit square $[0,1]^2$. Each point $X_i$ selects $\xi_i$ neighbors
at random, without replacement, among those points $X_j$ ($j\neq i$)
for which $\|X_i-X_j\| < r_n$, and the selected vertices are connected to
$X_i$ by an edge. The number $\xi_i$ of the neighbors is an integer-valued
random variable, chosen
independently with identical distribution for each $X_i$ such that $\xi_i$
satisfies $1\le \xi_i \le \kappa$ for a constant $\kappa>1$.
We prove that when $r_n = \gamma_n \sqrt{\log n/n}$ for $\gamma_n \to \infty$
with $\gamma_n =o(n^{1/6}/\log^{5/6}n)$, then the random geometric irrigation graph
experiences \emph{explosive percolation} in the sense that when
$\EXP \xi_i=1$, then the largest connected component has size $o(n)$
but if $\EXP \xi_i >1$, then 
the size of the largest connected component is with high probability $n-o(n)$.
This offers a natural non-centralized sparsification of a random geometric
graph that is mostly connected.
\end{abstract}




\section{Introduction}

We study the following model of random geometric ``irrigation'' graphs.
Let $\bX=\{X_1,\ldots,X_n\}$ 
\nomenclature[X]{$\bX$}{the set of $n$ i.i.d.\ points in $[0,1]^2$}
be a set of uniformly distributed random points in $[0,1]^2$. 
Given a positive number $r_n>0$,
\nomenclature[rn]{$r_n$}{the radius of visibility of the random geometric graph}
we may define the random geometric graph $G_n(r_n)$ with vertex set
$[n]:=\{1,\ldots,n\}$ in which vertex $i$ and vertex $j$ are connected
\nomenclature[\%]{$[n]$}{The set $\{1,2,\dots, n\}$}
if and only if the distance of $X_i$ and $X_j$ does not exceed the 
threshold $r_n$ \cite{Gilbert1961a,Pen03}. 
\nomenclature[Gnr]{$G_n(r_n)$}{$G_n(r_n)$ the random geometric graph with $n$ random points and 
radius of visibility $r_n$}
To avoid non-essential technicalities arising from irregularities around the borders 
of the unit square, we consider $[0,1]^2$ as a torus.
Formally, we measure distance of $x=(x_1,x_2),y=(y_1,y_2)\in [0,1]^2$ by 
$$d(x,y) = \left(\sum_{i=1}^2 \min(|x_i-y_i|, 1-|x_i-y_i|)^2 \right)^{1/2}.$$
It is well known that the connectivity threshold for the graph $G_n(r_n)$ 
is $r_n^\star=\sqrt{ \log n /(n\pi)}$
 (see, e.g., Penrose \cite{Pen03}). 
This means that, for any $\epsilon>0$,
$$
\lim_{n\to\infty}
\p{G_n(r_n) \text{~is connected}} 
=
\left\{
\begin{array}{l l}
0 & \text{~if~} r_n\le (1-\epsilon)r_n^\star\\
1 & \text{~if~} r_n\ge (1+\epsilon)r_n^\star.
\end{array}
\right.
$$
In this paper we consider values of $r_n$ well above the connectivity
threshold. So $G(r_n)$ is connected with high probability.
One may sparsify the graph in a distributed way by selecting, randomly, 
and independently for each vertex $u$, a subset of the edges adjacent to $u$, 
and then consider the subgraph containing those edges only. 
Such random subgraphs are sometimes called
\emph{irrigation graphs} or \emph{Bluetooth networks}
 \cite{FeMaPaPe04,DuJoHaPaSo05,DuHaMaPaPe07,CrNoPiPu09,PePiPu09}. 
A related model of \emph{soft random geometric graphs}, resulting from bond percolation 
on the geometric graph $G_n(r_n)$ is studied by \citet{Penrose2013a}. 
In this paper, we study the following slight generalization of the irrigation graph model:

\medskip
\noi\textsc{the irrigation graph.}\ 
We consider a positive integer-valued random variable $\xi$.
We assume that the distribution of $\xi$ is such that there exists
a constant $\kappa>1$ such that $\xi \in [1,\kappa]$ with probability one.
The \emph{random irrigation graph} $\Gamma_n=\Gamma_n(r_n,\xi)$ is obtained 
as a random subgraph of $G_n(r_n)$ as follows. 
For every $x\in [0,1]^2$, define $\rho(x)=|B(x,r_n)\cap \bX|$ to be the number of 
points of $\bX$ that are visible from $x$, where $B(x,r)=\{y\in [0,1]^2: d(x,y)< r\}$. 
\nomenclature[B]{$B(x,r)$}{The ball of radius $r$ centered at $x$}
\nomenclature[r]{$\rho(x)$}{The number of points of $\bX$ in the ball $B(x,r_n)$}
With every point $X_u\in \bX$, we associate $\xi_u$, an independent copy of the random 
variable $\xi$. Then given that $X_u\in \bX$ and $\xi_u$, 
let $\bY(X_u):=(Y_i(X_u), 1\le i\le \xi_u \wedge \rho(X_u))$ be a subset of elements 
of $\bX\cap B(X_u,r_n)$ chosen uniformly at random, without replacement. 
(Note that this definition allows a vertex to select itself. Such a selection
does not create any edge. In a slight modification of the model, the selection
is from the set $\bX\cap B(X_u,r_n)\setminus \{X_u\}$. 
Since self-selection is unlikely, all asymptotic results remain unchanged in the modified model.) 

We then define $\Gamma_n^+$ as the digraph on $[n]$ in which two vertices $u,v\in [n]$ are 
connected by an oriented edge $(u,v)$ if $X_v=Y_i(X_u)$ for some 
$1\le i\le \xi_u\wedge \rho(X_u)$. 
The out-degree of every vertex 
in $\Gamma_n^+$ is bounded by $\kappa$.
Finally, we define $\Gamma_n$ as the graph on $[n]$
in which $\{u,v\}$ is an edge if either $(u,v)$ or $(v,u)$ is an oriented 
edge of $\Gamma_n^+$. 

We study the size of the largest connected component of 
the random graph $\Gamma_n(r_n, \xi)$ 
for large values of $n$. We say that a property of the
graph holds \emph{with high probability} (whp) when
the probability that the property does not hold is bounded 
by a function of $n$ that goes to zero as $n\to \infty$.

\medskip
\noindent\textsc{connectivity of random geometric irrigation graphs.}\ 
Irrigation subgraphs of random geometric graphs
have some desirable connectivity properties. In particular,
the graph remains connected with a significant reduction of the number
of edges when compared to the underlying random geometric graph.
Connectivity properties of $\Gamma_n(r_n,c_n)$ (i.e., when
$\xi=c_n$ is deterministic, possibly depending on $n$) 
have been 
studied by various authors.  
\citet{DuJoHaPaSo05} showed that when $r_n=r>0$ is independent of $n$,
$\Gamma_n(r,2)$ is connected with high probability. 
Note that when $r> \sqrt{2}/2$ then the underlying random geometric graph
$G_n(r)$ is the complete graph and  
$\Gamma_n(r,2)$ is just
 the $2$-out random subgraph of the complete graph 
analyzed by Fenner and Frieze \cite{FeFr1982a}.
When $r$ is bounded away from zero, the underlying random geometric graph 
$G(r_n)$ is still an expander and $\Gamma_n(r,2)$ exhibits 
a similar behavior. Geometry only comes into play when $r_n\to 0$ 
as $n\to\infty$. In this regime, \citet{CrNoPiPu09}
proved that there exist constants $\gamma_1,\gamma_2$ such that
if $r_n \ge \gamma_1 \sqrt{\log n/n}$ and $c_n \ge \gamma_2 \log(1/r_n)$,
then $\Gamma_n(r_n,c_n)$ is connected with high probability. The 
correct scaling for the connectivity threshold for 
$r_n \sim \gamma \sqrt{\log n/n}$ for sufficiently large $\gamma$
was obtained by \citet{BrDeFrLu2011a} 
who proved that the connectivity threshold for the irrigation graphs 
with $r_n\sim \gamma \sqrt{\log n /n}$ is 
$$c_n^\star:=\sqrt{\frac{2\log n}{\log\log n}},$$
independently of the value of $\gamma$.
More precisely, for any $\epsilon \in (0,1)$, one has
\begin{equation}\label{eq:threshold_conn}
\lim_{n\to\infty}\p{\Gamma_n(r_n,c_n)~\text{is connected}}
=
\left\{
\begin{array}{l l}
0 & \text{if~} c_n \le (1-\epsilon) c_n^\star\\
1 & \text{if~} c_n \ge (1+\epsilon) c_n^\star.
\end{array}
\right.
\end{equation}
Thus, the irrigation 
subgraph of a random geometric graph 
preserves connectivity with high probability 
while keeping only $O(n c_n^\star)$ 
edges, which is much less than the
$\Theta(n\log n)$ edges of the initial random geometric graph. 
However, the obtained random irrigation subgraph
is not authentically sparse as the average degree still grows with $n$. 

One way to obtain connected sparse random geometric irrigation graphs is to 
increase the size $r_n$ of the ``visibility window'' slightly.
Indeed, we show elsewhere \cite{BrDeLu2014} that by taking
$r_n$ slightly larger, as $r_n \sim n^{-1/2+\epsilon}$ for some
fixed $\epsilon >0$, there exists a constant $c=c(\epsilon)$
such that $\Gamma_n(r_n,c)$ is connected with high probability.

\medskip
Otherwise, one needs to relax the constraint of connectivity, and see how this 
affects the graph. 
In this paper we study the emergence of a ``giant'' component
(i.e., a connected component of linear size) of random geometric
irrigation graphs when $r_n\sim \gamma\sqrt{\log n/n}$ for a sufficiently
large constant $\gamma$
(i.e., just above the connectivity threshold of the
underlying random geometric graph $G_n(r_n)$). The main result shows that already 
when $\EXP \xi > 1$, the graph $\Gamma_n(r_n,\xi)$ has a connected component 
containing almost all vertices. Interestingly,
there is not only a phase transition
around a critical value in the edge density but the phase
transition is \emph{discontinuous}. More precisely, we show that
when $\EXP \xi =1$ (or equivalently $\xi=1$, that is, when 
the average degree 
is about $2$), the largest
component of $\Gamma_n(r_n,\xi)$ is of size $o(n)$, while
for any $\epsilon>0$, if $\EXP \xi = 1+\epsilon$, then
with high probability, $\Gamma_n(r_n,\xi)$ has a component of size
$n-o(n)$. The phenomenon when there is a discontinuous phase
transition was coined ``explosive percolation''
and has received quite a lot of attention recently \cite{PaSpStTh11}.
In explosive percolation, the size of the largest component,
divided by the number of vertices, considered as a function
of the average degree, suffers a discontinuous jump
from zero to a positive value. In the present case we have 
even more: the jump is from zero to the maximal value of one.
Therefore, the random graph process experiences a 
\emph{super-explosive phase transition} or \emph{instant percolation}. 
The main results of the paper are summarized in the following theorems.

\begin{theorem}\label{thm:main}
Assume that $\EXP \xi > 1$.
For every $\varepsilon\in(0,1)$ there exists a constant $\gamma >0$ 
such that for
$r_n\ge \gamma \sqrt{\log n/n}$,
$$\p{\sC_1(\Gamma_n(r_n,\xi))\ge (1-\varepsilon) n} \xrightarrow[n\to\infty]{}1~,$$
where $\sC_1(\Gamma_n(r_n,\xi))$ denotes the size of the largest connected
component of the graph $\Gamma_n(r_n,\xi)$.
\nomenclature[C]{$\sC_1(G)$}{The number of vertices of the largest connected component of the graph $G$}
\end{theorem}

\medskip
We also prove that when $\xi=1$, with probability tending to one, 
the largest connected component of the irrigation graph $\Gamma_n(r_n,\xi)$ is sublinear:
\begin{theorem}\label{thm:sublinear_critical}
Suppose that $r_n=o((n\log n)^{-1/3})$. Then, for any $\varepsilon>0$
$$\p{\sC_1(\Gamma_n(r_n,1)) \ge \varepsilon n} \xrightarrow[n\to\infty]{}0.$$
\end{theorem}

The two theorems may be combined to prove the following ``instant-percolation'' result.

\begin{cor}
Suppose
$r_n / \sqrt{\log n/n}\to \infty$ and $r_n= o((\log n/n)^{1/3}$.
Then $\Gamma_n(r_n,\xi)$ 
experiences \emph{super-explosive percolation} in the sense that
\begin{itemize}
\item[(i)]
if $\EXP \xi\le 1$, then $\sC_1(\Gamma_n(r_n,\xi))=o(n)$ in probability;
\item[(ii)]
if $\EXP \xi>1$, then $n-\sC_1(\Gamma_n(r_n,\xi))=o(n)$ in probability.
\end{itemize}
\end{cor}

Note that for classical models of random graphs, including both random
geometric graphs \cite{Pen03} and Erd\H{o}s--R\'enyi random graphs
\cite{JaLuRu2000,Bollobas2001}, the proportion of vertices in the largest 
connected component is bounded away from one whp when the average degree is bounded 
by a constant. Furthermore, for these graphs, the size of the
largest connected component is continuous in the sense that the
(limiting) proportion of vertices in the largest connected component
vanishes as the average degree tends to the threshold value. The behavior of 
random geometric irrigation graphs
is very different, since the largest connected component contains
$n-o(n)$ vertices as soon as the expected degree is greater than 
two.
In particular, from a practical point of view, the irrigation graph provides an 
\emph{almost optimal} and \emph{distributed} algorithm for sparsification of the 
underlying graph. (Here distributed refers to the fact that every 
vertex makes its choices independently, as in distributed algorithms.)
Indeed the largest connected component contains $n-o(n)$ edges, and we achieve this with 
only $n(1+\epsilon)$ edges while any such graph must contain at least $n-o(n)$ edges.


Recall that when the underlying graph is the complete graph $K_n$ (i.e., when 
$r\ge \sqrt{2}/2$),  then the irrigation graph model corresponds to the $c$-out
graphs studied by \citet{FeFr1982a} if $\xi=c\in \N$ almost surely. 
With $c=1$, a random $1$-out subgraph of $K_n$ is a just a random mapping
\cite{Kolchin1986,FlOd1990a} (a uniformly random function from $[n]$
to $[n]$).  In particular, for $c=1$ and with high probability, a
random $1$-out subgraph of $K_n$ contains a connected component of
linear size but is not connected. For $c\ge 2$, a random $c$-out
subgraph of $K_n$ is 2-vertex and 2-edge connected with high
probability \cite{FeFr1982a}. (See also
\cite{Bender1974a,Bender1975a}: although a bit cryptic, Theorem~3
there applies to unions of two random mappings and shows that such graphs are
asymptotically connected.)  One may easily verify that if we write
$K_n(\xi)$ for the random irrigation subgraph of $K_n$ in which vertex
$i$ chooses $\xi_i$ random neighbors and $\Ec{\xi}>1$, then, as $n\to\infty$,
$$\p{K_n(\xi) \text{~is connected}}\to1~,$$
as an easy generalization of \citet{FeFr1982a}.

\section{Preliminaries: discretization and regularity of the point set}\label{sec:prelem}

The proof relies heavily on different levels of 
discretization of the torus into smaller sub-squares, as shown in
Figure \ref{fig:discretization}.
The largest of these sub-squares are called \emph{cells} and
have side length about $kr_n/2$ where $k$ is a fixed large odd natural number.
More precisely, let $k\ge 1$ be odd and define
\begin{equation}\label{eq:def_m-rp}
m:=\left\lceil \frac 2 {kr_n} \right\rceil 
\qquad \text{and}\qquad 
r_n':= \frac 2 {km}~.
\end{equation}
The unit square is then partitioned into $m^2$ congruent cells of side 
length $1/m=r_n'k/2$. Note that $(1-kr_n)r_n \le r_n' \le r_n$ for all
$n$ large enough.

A cell $Q$ is further partitioned into $k^2d^2$ square \emph{boxes}, 
\nomenclature[Q]{$Q$}{A cell}
each of side length $1/(mkd) = r_n'/(2d)$, for some natural number $d\ge 1$.
We make $d$ odd and call $C(Q)$ the central box of cell $Q$.
\nomenclature[C]{$C(Q)$}{The central box of cell $Q$}
A typical square box is denoted by $S$
\nomenclature[S]{$S$}{A square box}, and we let $\cS(Q)$
\nomenclature[S]{$\cS(Q)$}{The collection of boxes in the cell $Q$}
be the collection of all boxes in cell $Q$.

\begin{figure}
\centering
\includegraphics[scale=.7]{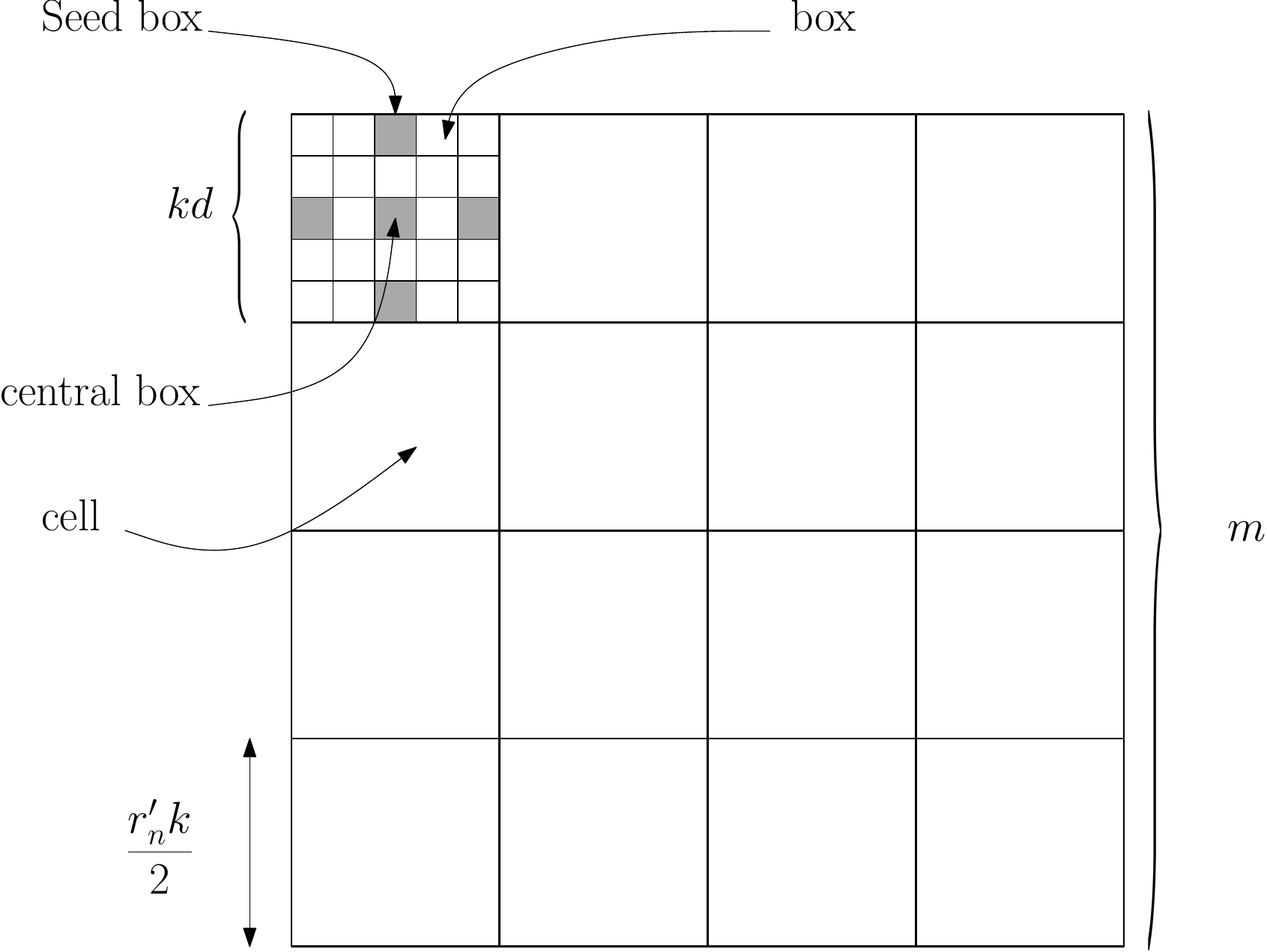}
\caption{\label{fig:discretization}
The different levels of discretization of the torus $[0,1]^2$ is shown
here with $k=5$, $d=1$, and $m=4$.
The torus is sub-divided into $m^2$ congruent squares, called cells, and each cell
is further divided into $k^2d^2$ small squares, called boxes.
The central box and the seed boxes of one of the cells are marked.
 }
\end{figure}

Note that there are two independent sources of
randomness in the definition of the random graph
$\Gamma_n(r_n,\xi)$. One comes from the random underlying geometric
graph $G_n(r_n)$ (i.e., the collection  $\bX$ of random points), and the
other from the choice of the neighbors of each vertex. 
We will always work conditionally on the locations of the points in $\bX$. 
The first step is to guarantee that, with high probability, the random set $\bX$ 
satisfies certain regularity properties. In the rest of the proof we assume that the 
point set $\bX$ satisfies the required regularity property, fix $\bX$ and work 
conditionally.


In the course of the proofs, we condition on the locations of the
points $X_1,\ldots,X_n$ and assume that they are sufficiently
regularly distributed. The probability that this happens is estimated
in the following simple lemma that relies on standard
estimates of large deviations for
binomial random variables.

Fix odd positive integers $k$ and $d$ and consider the partitioning of $[0,1]^2$
into cells and boxes as described above.
For a cell $Q$, and a box $S\in \cS(Q)$, we have
$$\Ec{|\bX \cap S|}=\frac{n}{(mkd)^2}=\frac {n r_n'^2}{4d^2}.$$
Fix $\delta\in (0,1)$. A cell $Q$ is called $\delta$-\emph{good} if for every 
$S\in \cS(Q)$, one has 
$$\frac{(1-\delta)nr_n^2}{4d^2}\le |\bX \cap S| \le \frac{(1+\delta)n r_n^2}{4d^2}.$$

\begin{lemma}\label{lem:strong_uniform}
For every $\delta\in (0,1)$, there exists $\gamma>0$ such that if 
$r_n\ge \gamma\sqrt{\log n/ n}$, then for all $n$ large enough, 
$$\inf_Q \pc{Q \text{~is $\delta$-good}} \ge 1 - 2(mkd)^2 n^{-\gamma^2\delta^2 / (24 d^2)}.$$
In particular, if $\gamma^2> 24d^2/\delta^2$ then 
$$ \lim_{n\to \infty} \pc{\text{every cell $Q$ is $\delta$-good}} = 1~.$$
\end{lemma}

See the Appendix for the proof.

\section{An overwhelming giant: Structure of the proof}\label{sec:roadmap}

\subsection{General approach and setting}\label{sec:gen_approach}

Our approach consists in exhibiting a large connected component by 
exposing the edges, or equivalently the choices of the points, in a specific order
so as to maintain a strong control on what happens. 
The general strategy has two phases: first a \emph{push}-like phase in which 
we aim at exposing edges that form a connected graph that is 
fairly dense almost everywhere; we call this subgraph the \emph{web}. 
Then, we rely on a \emph{pull}-like phase in which we expose edges from the points 
that are not yet part of the web and are trying to hook up to it. 

\medskip\noindent
\textsc{the push phase.}\ 
The design and analysis of the push phase is the most delicate part of the construction. 
It is difficult to build a connected component with positive density while keeping some 
control on the construction. For instance, following the directed edges in $\Gamma_n^+$
from a single point, say $x$, in a breadth-first manner produces an exploration of $\Gamma_n^+$ 
that \emph{resembles} a branching process. That exploration needs to look at  
$\Omega(\log\log n)$ neighborhoods of $x$ in order to reach the $\epsilon \log n$ 
total population necessary to have positive density in at least one ball of radius $r_n$. 
However, by the time the $\Theta(\log\log n)$ neighborhoods have been explored, 
the spread of the cloud of points discovered extends as far as $\Theta(\log \log n)$ away
from $x$ in most directions: in other words, doing this would 
waste many edges, and make it difficult to control the dependence between the 
events of reaching positive density in different regions of the square. 
An important consequence is that is it not reasonable to expect that two points 
that are close are connected locally: we will prove that points are indeed connected 
with high probability, but the path linking them does wander far away from them. 

To take this observation into account, in a first stage we only build a sort of skeleton 
of what will later be our large connected component. That skeleton, which we call the web
\emph{does not} try to connect points locally and its aim is to provide an almost ubiquitous 
network to which points will be able to hook up easily. The construction of this 
web uses arguments from percolation theory and relies on the subdivision of the unit square 
into $m^2$ cells described above. It is crucial to keep in mind that for the construction we 
are about to describe to work the web should be connected in a directed sense.

The cells define naturally an $m\times m$ grid as 
a square portion of $\Z^2$, which we view as a directed graph.
To avoid confusion with vertices and edges of the graph $\Gamma_n(r_n,\xi)$,
we call the vertices of the grid graph \emph{nodes} and its edges \emph{links}. 
More precisely, let $[m]:=\{1,\dots, m\}$. 
We then consider the digraph $\Lambda_m^+$ on the node set 
$[m]^2$ whose links are the pairs $(u,v)$ whose $\ell_1$ distance equals one;
the oriented link from $u$ to $v$ is denoted by $uv$.  
Write $E_m^+$ for the link set, so that $\Lambda^+_m=([m]^2, E_m^+)$. 

The construction of the web uses two main building blocks: 
we define events on the nodes and the links of $\Lambda_m^+$ such that 
\begin{itemize}
\item a \emph{node event} is the event that, starting from a vertex
in the central box of the cell, if one tracks the selected neighbors
of the vertex staying in the cell, then the selected neighbors of 
these neighbors in the cell and so on up to a number of hops $k^2$, 
then the resulting component populates the cell in a uniform manner 
-- see Proposition \ref{pro:vertex_proba} for the precise statement;
\item a \emph{link event} allows the connected component built within the cell to 
propagate to a neighboring cell. 
We  show that both node events and link event happen with high
probability. See Lemma~\ref{lem:diedge}
\end{itemize}
It is important to emphasize that in proving that node and link events
occur with high probability, we make use of a coupling with suitably
defined branching random walks that are independent of the precise
location of the points at which such events are rooted. 
Although this does not suffice to make all node and link events become independent,
it helps us control this dependence and allows 
us to set up a joint site/bond percolation argument on $\Z^2$ that proves 
the existence of a directed connected component that 
covers most cells. The node and link events are described precisely and the bounds on 
their probabilities are stated in Sections~\ref{sec:node} and~\ref{sec:edge}, 
respectively.
Finally, in Section~\ref{sec:sketch_layer}, we show how to combine the node and 
link events in order to construct the web using a coupling with a percolation process. 
The proof relating to the estimates of the probabilities of the node events is rather
intricate, and we present them in Section~\ref{sec:vertex}.

\medskip
\noindent\textsc{the pull phase.}\ 
The analysis of the pull phase relies on proving that any vertex
not yet explored in the process of building the web is in the same component as the web,
with high probability. In order to prove this, one may construct another
web starting from such a vertex, which succeeds with high probability
 by the arguments of the push phase.
Then it is not difficult to show that the two webs are connected with 
high probability.
The details are developed 
in Section~\ref{sec:pull_phase}.

\subsection{Populating a cell: node events}
\label{sec:node}

In proving the existence of the web (i.e., a connected component 
that has vertices in almost every cell), we fix $\delta>0$ and any point set $\bX$ 
for which every box is $\delta$-good and work conditionally. 
Thus, the only randomness comes from the choices of the edges.
We reveal edges of the digraph
$\Gamma^+_n$ in a sequential manner. In order to make sure that
certain events are independent, once the $\xi_i$ out-edges of a
vertex $X_i$ have been revealed, the vertex becomes \emph{forbidden} and 
excluded
from any events considered later. We  keep control of the number
and density of forbidden
points during the entire process.

In Section \ref{sec:sketch_layer} we describe the order in which cells 
are examined. In this section we look into a single cell $Q$ and describe an
event ---the so-called ``node event''--- that only depends on edge choices 
of vertices within the cell. All we need is a starting vertex
$x\in \bX$ in the central box $C(Q)$ of the cell and a set $F$ of forbidden vertices.
Both $x$ and $F$ may depend on the evolution of process before the cell
is examined. However, by construction (detailed below), we guarantee that the
set of forbidden vertices $F$ only has a bounded number of elements in each box, 
and therefore does not have a significant impact on the outcome of the node event.
Similarly, the starting vertex $x$ originates from an earlier stage 
of the process but its exact location is unimportant, again by the definition of the 
node event, as detailed below.




Consider a cell $Q$ and a point $x\in \bX\cap Q$. For $i\ge 0$, let $\tilde \Delta_x(i)$ 
denote the 
collection of points of $\bX\cap Q$ that can be reached from $x$ by following $i$ directed 
edges of $\Gamma_n^+$ without ever using a point lying outside of $Q$. 
Let $F\subset \bX$ denote the set of \emph{forbidden} points containing the $y\in \bX$ 
whose choices have already been exposed.
Let $\Delta_x(i)$ be the subset of points of $\tilde \Delta_x(i)$ that can be 
reached from $x$ without ever using a point in~$F$. 

Recall that the cell $Q$ is partitioned into $k^2d^2$ square \emph{boxes}
of side length $r_n'/(2d)$ and that $\cS=\cS(Q)$ denotes the collection 
of boxes of $Q$. 
The next proposition shows that, with high probability,
any cell $Q$ with starting point $x\in \bX\cap Q$ 
is such that $\Delta_x(k^2)$ populates $Q$ in the following way: 
for every $S\in \cS(Q)$, we have $|\Delta_x(k^2)\cap S|\ge \Ec{\xi}^{k^2/2}$. 
We refer to this event as $N_x(Q)$ and the corresponding local 
connected component is called a \emph{bush}. The proof is delayed 
until Section~\ref{sec:vertex}.

\begin{proposition}\label{pro:vertex_proba}
For a cell $Q$ and vertex $x\in \bX\cap C(Q)$, define the \emph{node event} 
\[
N_x(Q) = \left\{ \forall S\in \cS(Q): |\Delta_x(k^2) \cap S| \ge \Ec{\xi}^{k^2/2}
    \right\}~.
\]
\nomenclature[N]{$N_x(Q)$}{The node event in cell $Q$ started from $x\in C(Q)$}
For every $\eta>0$, there exist constants $d_0,k_0,n_0\ge 1$, and $\delta_0>0$ such 
that, provided that the cell is $\delta$-good and that 
$\sup_{S\in \cS(Q)}|F\cap S|<\delta_0 n r_n^2$, then
for all $k\ge k_0$, $d\ge d_0$, $n\ge n_0$ and for all $x\in \bX\cap C(Q)$, 
$$\pc{N_x(Q)\,|\,\bX} \ge 1-\eta~.$$
\end{proposition}

We note that although it may seem that the events $N_x(Q)$, for distinct cells $Q$ 
would be independent (they depend on the choices of disjoint sets of vertices), it is not 
the case. It will become clear later that the events do interact through the 
set $F$, which will be random, but that this dependence can be handled.

\subsection{Seeding a new cell: link events}\label{sec:edge}

We define an event that permits us to extend a bush confined to a cell $Q$ and 
to find a directed path from it to a point $x'$ in the central box of a
neighboring cell $Q'$.

For a given cell $Q$, the set of $(kd)^2$ boxes $\cS(Q)$ is naturally indexed by 
$$\{-\flr{kd/2}, \dots, \flr{kd/2}\}^2.$$ 
Among the boxes $S\in \cS(Q)$, 
let $\cI(Q)$ 
\nomenclature[I]{$\cI(Q)$}{The collection of four seed/infection boxes of the cell $Q$} 
denote the collection of the four boxes that correspond to the
coordinates $(-\flr{kd/2}, 0)$, $(\flr{kd/2}, 0)$, $(0, -\flr{kd/2})$ and $(0, \flr{kd/2})$ 
(Figure \ref{fig:discretization}).
It is from points in these boxes that we try to ``infect'' neighboring cells, and we refer to 
these boxes as \emph{seed boxes} or \emph{infection boxes}. 
For two adjacent cells $Q$ and $Q'$, we let $I(Q,Q')$
\nomenclature[I]{$I(Q,Q')$}{The infection box in $Q$ on the face shared by $Q$ and $Q'$} 
denote the seed box lying in $Q$ against the face shared by $Q$ and $Q'$. 
Suppose, as before, that there is a set $F\subset \bX$ of forbidden points. 
For a point $y\in \bX\cap I(Q,Q')$, let $\Delta^\circ_y(i)$ denote the points of $\bX$ that 
can be reached from $y$ using $i$ directed edges of $\Gamma_n^+$ 
without using any point lying outside of $Q'$ or in $F$, except for $y$ itself. 
Let $J_y(Q,Q')$ be the event that $\Delta_y^\circ(\cel{kd/2})$ contains a point lying in 
the central box of $Q'$: 
$$J_y(Q,Q')=\{\Delta_y^\circ(\cel{kd/2})\cap C(Q')\ne \varnothing\}.$$
Then, for $R\subset \bX\cap I(Q,Q')$, we let $J_R(Q,Q')=\cup_{y\in R} J_y(Q,Q')$.
The event $J_R(Q,Q')$ is called a \emph{link event}.


\begin{lemma}\label{lem:diedge}
Let $Q$ and $Q'$ be two adjacent cells. Suppose that $Q'$ is $\delta$-good 
for $\delta\in(0,1/4)$ and that 
$\sup_{S\in \cS(Q)} |F\cap S|<\delta_0 n r_n^2$ for $\delta_0<1/(4d)^2$. 
Then, for every $k$, and $d$ there exists $n_0$ such that for every $n\ge n_0$,
and for any set $R\subseteq \bX \cap I(Q,Q')$, we have 
$$\pc{J_R(Q,Q')\,|\,\bX} \ge 1 - \exp\pran{- \frac{|R|}{(10\beta d^2)^{kd}}}~,$$
where $\beta = (1+\delta) (1/(2d) + (k/(16d^2))$.
\end{lemma}

\begin{proof}
Write $h=\cel{kd/2}$. 
Let $L_{0}=I(Q,Q'), L_{1},\dots, L_{h}=C(Q')$ denote the sequence of boxes on the straight 
line from $I(Q,Q')$ to $C(Q')$. For $J_R(Q,Q')$ to occur, it suffices that for some $y\in R$,
one has $|\Delta_y^\circ(i)\cap L_i|\ge 1$, for every $i=1,\dots, h$; call $E_y^\circ$ the 
corresponding event. The $E_y^\circ$, $y\in R$, are not independent because the sets 
$\Delta_y^\circ(i)$, $i\ge 1$, might not be disjoint.  
However, on $\{\cap_{y\in R}\Delta_y^\circ(i)= \varnothing\}$, the events $E_y^\circ$, $y\in R$, 
are independent. 
Consider the ordering of the points in $R$ induced by the ordering in $\bX$, 
and write $x< x'$ if $x=X_i$ and $x'=X_j$ for $i<j$. 
To simplify the proof, we only consider a single path from any given point $y\in R$. 
Consider the path defined by $P_0(y)=y$, and for $i\ge 1$, $P_i(y)=Y_1(P_{i-1}(y))$;
this is well defined since $\xi_i\ge 1$ with probability one. 
Let $E_y$ be event that for every $i=1,2,\dots, h$, one has $P_i(y)\in L_i$. 
Then we have $E_y\subset E_y^\circ$.

Note that since $Q'$ is $\delta$-good, for any box $S\in \cS(Q)$ we have:
$$|\bX \cap S |\ge \frac{(1-\delta)nr_n^2}{4d^2}
\qquad \text{and}\qquad 
|S\cap \bX \cap F^c|\ge \frac{nr_n^2}{8d^2}$$ 
since $\delta<1/4$ and $\delta_0<1/(4d)^2$. 
Furthermore, at most $|R|h$ of the points of $\bX\cap S \cap F^c$ 
lie in $\cup_{y\in R} P_i(y)$, for some $i=1,2,\dots, h$. 
Now, for $y\in R$, let $\tau_y:=\inf\{i\ge 1: P_i(y) \not\in L_i\}$. 
Let $\cG_{y}^-$ be the sigma-algebra generated by $\{P_i(y): 0 \le i< \tau_y\}$, $y'<y$. 
Since every cell is $\delta$-good, for every $x\in [0,1]^2$, we have 
$|\bX \cap B(x,r_n)|\le \beta n r_n^2$.
Then, for every $y\in R$, and all $n$ large enough, we have
\begin{equation}\label{eq:seed_one-point}
\Cprob{E_y}{\cG_{y}^-}
\ge \pran{\frac 1 {10 \beta d^2}}^h
\ge \pran{\frac 1 {10 \beta d^2}}^{kd}=:\eta,
\end{equation}
where the last expression serves as the definition for the constant $\eta$. 
Here we used the fact that $10 \beta d^2 \ge 1$.
It follows that $|\{y\in R: E_y\}|$ dominates a binomial random variable $\bin(|R|;\eta)$
with parameters $|R|$ and $\eta$:
\begin{align*}
\pc{\exists y \in R: E_y^\circ}
\ge \p{\exists y\in R: E_y} 
&\ge \pc{\bin(|R|; \eta)>0} \\
&\ge 1 - e^{-\eta |R|}.
\end{align*}
Replacing $\eta$ by its expression in \eqref{eq:seed_one-point} yields the claim.
\end{proof}

\subsection{Building the web: the percolation process}
\label{sec:sketch_layer}

In this section, percolation arguments are used to show how the node and link events 
can be used to build the ``web'', a connected component that visits most
cells. The construction is based on an algorithm 
to decide which edges to expose depending on what we have 
seen so far. Once again, fix a point set $\bX$ such that every cell 
is $\delta$-good and work conditionnally. 

\medskip
\noi\textsc{defining a partial percolation configuration on the square grid of cells.}\ 
We encode an exploration process on the digraph $\Lambda^+_m$ of cells by defining a 
\emph{partial} and \emph{joint site/bond} percolation process using the 
node and link events defined above. 
At the same time, we keep track of the set of 
forbidden vertices $F$ discussed in Sections~\ref{sec:edge} and~\ref{sec:vertex}. 

For $u\in [m]^2$, we let $Q_u\subset [0,1]^2$ denote the corresponding cell.
The nodes of $[m]^2$ are ordered lexicographically: 
for $u=(u_1,u_2)$, $v=(v_1,v_2)$ we write $u\preceq v$ if $u_1\le v_1$ or 
if $u_1=v_1$ and $u_2\le v_2$. 
We proceed with an exploration process in the lexicographic order, maintaining,
at every step $i=0,1,2,\ldots$ of the process,
a partition of $[m]^2$ into three sets of nodes $[m]^2=A_i\cup E_i\cup U_i$, where 
we call the nodes in $A_i$ \emph{active}, the ones in $E_i$ \emph{explored}, and 
those of $U_i$ \emph{unseen}. Initially, all the nodes are unseen and therefore 
$U_0=[m]^2$, $A_0=\varnothing$, $E_0=\varnothing$. 
The sets $A_i, E_i, U_i$, $i\ge 0$, are designed in such a way that, 
\begin{itemize}
\item at any time $i\ge 0$, any node $u\in A_i$ has a distinguished vertex $x_u$
in the center box $C(Q_u)$ for which we can check if 
the node event $N_{x_u}(Q_u)$ (defined in Proposition \ref{pro:vertex_proba}) 
occurs; the set of forbidden vertices that is used 
to assess this event is $F_i$ to be defined shortly. 
\item The nodes $u\in E_i$ are the ones that have been active from some time 
$j<i$ and for which the node event $N_{x_u}(Q_u)$ has already been 
observed. 
\end{itemize}

We now move on to the precise description of the algorithm and of the sets 
$A_i, E_i, U_i\subset [m]^2$, and $F_i\subset \bX$. Initially, we set $F_1=\varnothing$. 
Then we proceed as follows, for $i\ge 1$. If $E_i=[m]^2$, then we have already 
``tested'' a node event for each node and we are done, and we now suppose that
$E_i\ne [m]^2$. Then, there must be some node in either $A_i$ or $U_i$.

{\bf (i.)}\ Suppose first that $A_i\ne \varnothing$. Then, let $u_i$ be the node of 
$A_i$ that is lowest in the lexicographic order. By construction, 
there is a distinguished vertex $x_{u_i}\in C(Q_{u_i})$.
Say that the node $u_i$ is \emph{open} and set $\tilde \sigma(u_i)=1$ 
if the node event $N_{x_{u_i}}(Q_i)$ succeeds. 
If this is the case, all four seed boxes in $Q_{u_i}$ contain a set of points 
of the bush constructed in $Q_{u_i}$ of cardinality at least $\Ec{\xi}^{k^2/2}$ 
that are all connected to $x_{u_i}$ within $Q_{u_i}$. 
Consider all the oriented links $u_iv$, where $v\in U_i$, and let 
$R_{u_iv}$ be the set of points that are lying in the 
seed box $S(Q_{u_i},Q_v)$ of $Q_{u_i}$ that is adjacent to $Q_{v}$. For any such link $u_iv$, 
we declare the oriented link open and set 
$\tilde \sigma(u_iv)=1$ if the link event $J_{R_{u_iv}}(Q_{u_i}, Q_v)$ 
(defined just before Lemma \ref{lem:diedge}) succeeds. 
(Note that we liberally use the notation $\tilde \sigma(\cdot)$ to indicate either
openness of a node $u$ by $\tilde\sigma(u)$ or the openness of an oriented link $uv$
by $\tilde\sigma(uv)$).

Let $V_i=\{v\in U_i: \tilde \sigma(u_iv)=1\}$. 
For every $v\in V_i$, since $J_{R_{u_iv}}(Q_{u_i},Q_v)$ succeeds, we have, by construction, 
a non-empty set of points of the center box $C(Q_v)$ that are connected to $R_{u_i,v}$ 
by directed links in $\Gamma_n^+$; we let $x_v$ be the one of these points that has the 
lowest index in $\bX$. Then, update the sets by putting $E_{i+1}=E_{i} \cup \{u_i\}$, 
$A_{i+1}=A_i \cup V_i \setminus \{u_i\}$, $U_{i+1}=U_i \setminus V_i$. 
As for the set of forbidden vertices, let $f_{i+1}$ be the collection of points $X_u\in \bX$ 
whose choices $Y_j(X_u)$, $1\le j\le \xi_u$, have been exposed when determining the 
node event $G_{x_{u_i}}(Q_{u_i})$ and the potential following link events. 
Then, let $F_{i+1}=F_i \cup f_{i+1}$.

{\bf (ii.)}\ If, on the other hand, $A_i=\varnothing$, then $U_i\ne \varnothing$.
Note that if this happens, it means that we have not succeeded in finding 
a point $x\in C(Q_{u_i})$ that is connected to the points previously explored 
and we need to start the exploration of a new connected component of $\Gamma_n^+$. 
Let $u_i\in U_i$ be the node with lowest lexicographic order. Then, 
set $A_{i+1}=\{u_i\}$, $E_{i+1}=E_i$ and $U_{i+1}=U_i\setminus \{u_i\}$. 
We then let $x_{u_i}$ be the point of $\bX \cap C(Q_{u_i})$ that has the 
lowest index in $\bX$. Such a point exists by the assumption 
of $\delta$-goodness and because the number of forbidden points in 
each cell is bounded (see Lemma \ref{lem:forbidden-points} below).

Note that the distinguished point $x_u$ of a cell $Q_u$ is chosen when 
the corresponding vertex is activated, which happens once and only once for 
every node $u\in [m]^2$. 

In order to use Proposition~\ref{pro:vertex_proba} and Lemma~\ref{lem:diedge}
for estimating the probability of node events and link events,
we need to make sure that the number of forbidden points stays 
under control. 

\begin{lemma}\label{lem:forbidden-points}
If $k$ is sufficiently large, then
for every cell $Q$, during the entire process, we have
$$|F\cap Q |\le \kappa^{2 k^2}~.$$
\end{lemma}
\begin{proof}
To reveal a node event $N_x(Q)$, one only needs to expose $\Delta_x(k^2)$ for a 
single point $x\in Q$. This requires to look at the edge choices of at most $k^2 \kappa^{k^2}$ 
vertices, all of which lie in $Q$. 
The points exposed during the evaluation of a link event account for 
a total of at most $4 \cdot k^2 \kappa^{k^2} \cdot k d \kappa^{kd}$. 
The claim follows easily. 
\end{proof}

\medskip
\noi\textsc{completing the percolation configuration.}\ 
Once the exploration 
process is finished, every node has been declared open or not, 
and we have assigned a value to every $\tilde\sigma(u)$, $u\in [m]^2$. 
\nomenclature[st]{$\tilde \sigma$}{A mixed site/bond percolation configuration where 
the bonds/sites are not independent}
However, we have not defined the status of all the 
oriented links $uv$. See Figure \ref{fig:OrientedClustersTorus}. 
In particular, $\tilde\sigma(uv)$ has only been
defined if $\tilde\sigma(u)=1$ and if $\tilde \sigma(v)$ had not been set to one before. 
For every oriented link $uv$, let $\theta(uv)$ be the indicator that 
a link event has been observed for $uv$. 
Let $H_m^+$ denote the open subgraph of $\Lambda_m^+$,  that consists of 
nodes $u$ and directed links $uv$ for which $\tilde\sigma(u)=1$ and $\tilde \sigma(uv)=1$, 
respectively.
A subset $K$ of nodes in $[m]^2$ is called an \emph{oriented connected component}
of $H_m^+$ if $\tilde \sigma(u)=1$ for all $u\in K$ and for all $u,v\in K$ there is
an oriented open path between $u$ and $v$, that is, a sequence 
$u=u_1,u_2,\ldots,u_\ell=v$ of nodes in $K$ such that $\tilde \sigma(u_iu_{i+1})=1$
for all $i=1,\ldots,\ell-1$.

\begin{figure}
\centering
\includegraphics[scale=1]{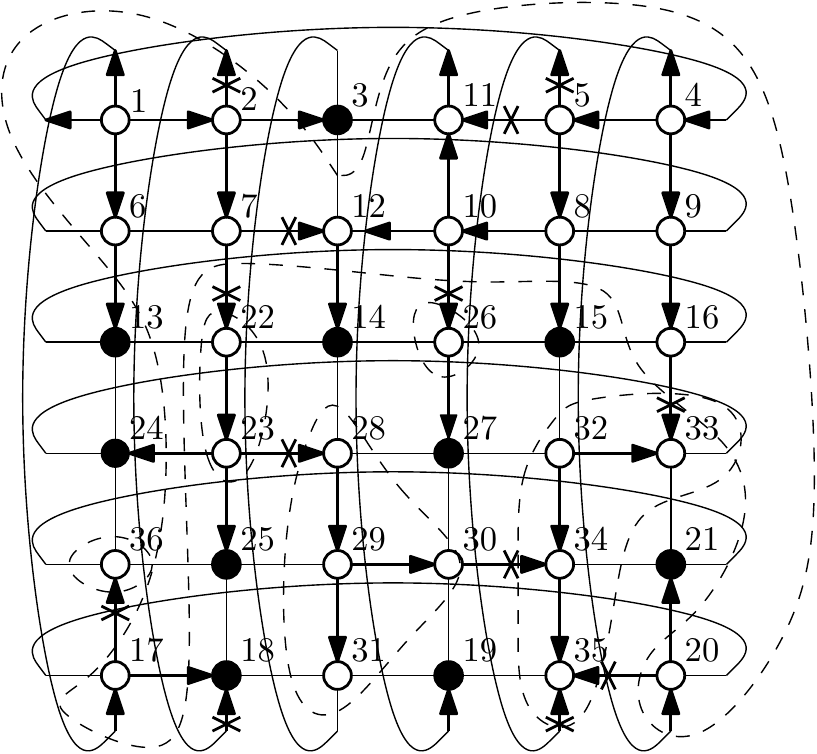}
\caption{\label{fig:OrientedClustersTorus}
The partial percolation configuration after exploring all node and link events
with the obtained oriented connected components. White nodes are those
for which the node event $N_x(Q)$ succeeds. Crossed arrows represent 
failed link events. The numbers near the nodes indicate the order in which the 
node events are tested. }
\end{figure}

In order to prove that $H_m^+$ contains an oriented connected component
containing most nodes ---and therefore proving the existence of the web---,
we embed $H_m^+$ in an \emph{unoriented} complete mixed site/bond percolation configuration 
on the digraph $\Lambda_m^+$. Then we  use results from the theory
of percolation to assert the existence of a connected component
containing most vertices.

In a general mixed site/bond percolation configuration, every node is open
independently with a certain probability $p$, and every undirected link is also
open independently with probability $q$.
Fix $\eta\in (0,1)$ and choose the parameters $k,d,\lambda$, and $\delta$
such that each node event occurs with probability at least $1-\eta$
and each (oriented) link event occurs with probability at least $1-\eta/2$.
In order to define the mixed site/bond percolation configuration,
we first assign states to the oriented links $uv$ for which 
$\theta(uv)=0$. 
Let $(\tilde \sigma(uv): uv\in E_m^+, \theta(uv)=0)$ be 
a collection of i.i.d.\ Bernoulli random variables with success probability $1-\eta/2$.
Now we declare an unoriented link $uv$ open if $\tilde\sigma(uv)=\tilde \sigma(vu)=1$.
Observe that although we have assigned a configuration 
$(\tilde\sigma(u), u\in [m]^2; \tilde \sigma(e), e\in E_m^+)$ to the digraph $\Lambda_m^+$, 
this collection of random variables is not independent. For instance, for any two nodes, 
the $\tilde \sigma(u)$ and $\tilde \sigma(v)$ are dependent for they interact through the 
set of $F$ of forbidden nodes which is random. However, 
\begin{align*}
\inf_{1\le i\le m^2} &\p{\tilde \sigma(u_i)=1\,|\, \bX,F_i}\ge 1-\eta\quad \text{and}\\
\inf_{1\le i\le m^2} \inf_{u_iv\in E_m^+} & \p{\tilde \sigma(u_iv)=1\,|\, \bX,F_i}\ge 1-\eta/2~,
\end{align*}
which implies that there exists two independent collections of i.i.d.\ Bernoulli random 
variables $(\sigma(u), u\in [m]^2)$, and $(\sigma(uv), uv\in E_m^+)$ with success probabilities
$1-\eta$ and $1-\eta/2$, respectively such that almost surely $\tilde \sigma(u)\ge \sigma(u)$ 
for $u\in [m]^2$ and $\tilde \sigma(uv)\ge \sigma(uv)$ for $uv\in E_m^+$. 
The configuration defined by $\sigma$ is a proper mixed site-bond percolation configuration.
\nomenclature[s]{$\sigma$}{A proper mixed site/bond percolation configuration}

By construction, in the configuration $\sigma$, every node and every unoriented link of $[m]^2$ 
is equipped with an independent Bernoulli random variable 
with success probability at least $1-\eta$.
Each node and each link is open if the corresponding Bernoulli variable
equals $1$. This is the \emph{mixed site/bond percolation} model 
considered, for example, by \citet{Ham80}. 
In such a configuration, we say that 
two nodes $u,v\in [m]^2$ are bond-connected in the configuration 
if there exists a sequence of open nodes $u=u_1,u_2,\dots, u_\ell=v$ 
for which every link $u_iu_{i+1}$, $1\le i<\ell$ is also open.
This equivalence relation naturally defines bond-connected components. 
Clearly, each bond connected component is also an oriented open 
component in $H_m^+$ and therefore it suffices to show that
the mixed site/bond percolation configuration has a
bond-connected component containing almost all nodes, with high probability.

In order to prove this, we use results of \citet{DePi1996a} for 
high-density site percolation by reducing the mixed site/bond
percolation problem to pure site percolation as follows:

\begin{lemma}
Consider mixed site/bond percolation on $[m]^2$ as defined above
where each node is open with probability $p$ and each link is open
with probability $q$, independently. The size of the largest
bond-connected component is stochastically dominated by
the size of the largest open component in site percolation on $[m]^2$
where each node is open with probability $pq^2$.
\end{lemma}

\begin{proof}
Split each link in the mixed model into two half-links,
and let each half-link be independently open with probability $\sqrt{q}$.  
We say that a link is open if both half-links are open.

Next, for a node $v$ in the mixed model, we call event $D(v)$
the event that the node and its four adjacent half-links
are open.  This occurs with probability $r:=pq^2$.
Now, consider a coupled site percolation model, also on the $m\times m$ torus,
in which the node $v$ is open if $D(v)$ occurs.  These
are independent events.
So, we have a site percolation model with node probability $r=pq^2$.
It is clear that if a path exists in the site percolation model then
a path exists in the mixed model, so the mixed model
percolates (strictly) better.  
\end{proof}

Now it follows from  \cite{DePi1996a} that in our mixed site/bond
percolation model where nodes and links are open with probability at least
$1-\eta$, the following holds: for every $\epsilon>0$ there
exists $\eta>0$ such that for all $m$ large enough,
\[
  \p{\text{there exists a bond-connected component of size } > (1-\epsilon)m^2} 
  > 1- \epsilon~.
\]
Now, for us the constant $\eta$ is controlled by $k,d, \delta$ and $\gamma$. 
Putting everything together, we have proved the existence of the web:

\begin{proposition}
\label{prop:web}
Let $\epsilon>0$. There exist $k_0,d_0,\delta,\gamma$ such that if $k>k_0$, $d>d_0$,
and $r_n>\gamma \sqrt{\log n/n}$, then for all $n$ large enough,
if all cells are $\delta$-good, then, with probability (conditional on $\bX$)
at least $1-\epsilon$, there exists a connected component of $\Gamma(r_n,\xi)$
such that at least $(1-\epsilon)$-fraction of all boxes contain at least $\EXP[\xi]^{k^2/2}$
vertices of the component.
\end{proposition}

\subsection{Finale: gathering most remaining points}
\label{sec:pull_phase}

In the previous sections we saw that after exploring only a constant 
number of points per cell (at most $m^2\kappa^{2 k^2}\le 2\kappa^{2 k^2}/(k^2r_n^2)$ 
in total by Lemma \ref{lem:forbidden-points})
with high probability, we can construct a connected component ---the so-called web--- 
of the graph $\Gamma(r_n,\xi)$ that contains at least $\EXP[\xi]^{k^2/2}$ points
in a vast majority of boxes. Recall that each box is a square of side length
$r_n'/(2d)$ where $d$ is a fixed but large odd integer.

It remains to prove that most other vertices belong to the same component
as the web, with high probability. To this end, first we show that 
any not yet explored vertex is contained in the same component as the web,
with high probability. As before, we fix a sufficiently small $\delta>0$ and fix a point 
set $\bX$ such that every cell is $\delta$-good.
Suppose that the exploration process of the previous sections has been 
carried out, revealing the edge choices of at most $\kappa^{2 k^2}$
points per cell (and thus also at most this many per box). If $n$ is so large that 
$\delta^2 \gamma \log n/(4d^2)> \kappa^{2 k^2}$, then even after removing
all vertices whose choices have already been exposed, every cell remains $2\delta$-good.
Let $x_i\in \bX$ be one of the still unseen vertices. We shift the coordinate system
so that the box containing $x_i$ becomes the central box of the first cell. 
Since the boxes in the new coordinate system were also boxes in the original 
coordinates, every cell is still $2\delta$-good in the new system.

Now we build a second web, with the aim that it contains $x_i$ with probability 
close to one.
We start the same exploration process from the vertex 
$x_i$ as in the construction of the web but now we place all vertices of the first 
web in the set of forbidden points. If $\delta$ is sufficiently small, then
Proposition \ref{prop:web} applies and, with probability at least
$1-\epsilon$, we obtain another web that has at least $\EXP[\xi]^{k^2/2}$
vertices in at least $(1-\epsilon)$-fraction of the boxes. The newly built web
may not contain the vertex $x_i$. However, by the homogeneity of the
mixed site/bond percolation process, each cell is equally likely to
be contained in the newly built web and therefore, with probability at least
$(1-\epsilon)^2$ vertex $x_i$ is contained in a component 
that has at least $\EXP[\xi]^{k^2/2}$
vertices in at least a $(1-\epsilon)$-fraction of the boxes. 
It is clear from the proof of Proposition~\ref{pro:vertex_proba} that,
in fact, each of these boxes contains at least $\EXP[\xi]^{k^2/2}$ points
whose edge choices have not been revealed in the process of building
the second web.
Thus, with probability at least $(1-\epsilon)^3$,
at least $(1-2\epsilon)$-fraction of the boxes contain at least $\EXP[\xi]^{k^2/2}$
points of the first web and at least $\EXP[\xi]^{k^2/2}$ points of the second web
that contains the vertex $x_i$. Now we may reveal the edge choices
of the vertices of the second web that have not been explored. 
Since the diameter of a box is less than $r_n$, the probability that
the two webs do not connect ---if they have not been connected already---
is at most
\[
  \left(1-\frac{\EXP[\xi]^{k^2/2}}{(1-2\delta)nr_n^2/(4d^2)}\right)^{m^2(1-2\epsilon)}
   = o(1)
\]
whenever $r_n=o(n^{-1/4})$.

Thus, conditionally on the fact that all cells are $\delta$-good, 
the expected number of vertices that do not connect to the web is $o(n)$.
Finally, Markov's inequality, and Lemma \ref{lem:strong_uniform}
complete the proof of Theorem~\ref{thm:main}.

\section{Getting out of the central box: Proof of Proposition~\ref{pro:vertex_proba}}
\label{sec:vertex}

\subsection{Constructing a branching random walk}

Most of the work consists in estimating the probability of the local events, 
while ensuring independence. 
In this section, we consider a single cell $Q$ of side length $kr_n'/2$. 
As we have already explained, the local bushes are constructed by a process that 
resembles a \emph{branching random walk} in the underlying geometric graph. 
The main differences with an actual branching random walk are that 
\begin{itemize}
\item the potential individuals are the elements of $\bX$, and so they are 
fixed conditionnaly on $\bX$), 
\item an element of $\bX$ only gets to choose its neighbors once; 
in particular, if a vertex $X_i$ is chosen that has already 
used up its $\xi_i$ choices, the corresponding branch of the exploration must stop 
(if we were to continue the exploration, it would trace steps that have already 
been discovered). 
\end{itemize}

The entire argument in this section is conditional on the location of the
points, assuming the regularity property that the cell $Q$ is $\delta$-good.
Recall that a cell $Q$ of side-length $kr_n'/2$ is called 
$\delta$-good if the numbers of points within every box $S$ it contains lies 
within a multiplicative $[1-\delta, 1+\delta]$ range of its expected value $\EXP |\bX\cap S|$. 
By Lemma~\ref{lem:strong_uniform}, for any $\epsilon>0$, the probability that
every cell is $\delta$-good is at least $1-\epsilon$ for all $n$ large enough
(provided the constant $\gamma$ in $r_n=\gamma \sqrt{\log n/ n}$ is large enough).

Fix a cell $Q$, in which we want to analyze the node event. Then, for every $i\in [n]$ 
such that $X_i\not \in Q$, we work with an independent copy of $\xi_i$. Since such
points $X_i$ are not considered when analyzing the node event on $\bX$, this has 
no effect on the event $N_x(Q)$, for $x\in \bX \cap C(Q)$. However, this makes the 
proofs a little smoother since we can look at all $k^2$ neighborhoods without 
worrying (at least in a first stage) whether the points are in $Q$ or not. 
Note the important fact that this is only used for the \emph{analysis}, and that the 
actual exploration is not carried out when the points leave the cell $Q$. In 
particular, this thought experiment does not affect the number of forbidden 
vertices.


\medskip
\noi\textsc{discretizing the steps.}\
We define a \emph{skimmed} version of the neighborhood exploration in which we drop some 
of the points in order to guarantee simplified dynamics. The simplification uses the 
refined discretization of the space into boxes.
Let $\cS$ denote the collection of all boxes (open squares). 
For a point $x\in [0,1]^2$, we let $S(x)\subset Q$ denote the box containing the point $x$ 
(this is well-defined for every point of $Q$ with probability one).


For a given point $x\in Q\subseteq [0,1]^2$, let $A^\circ_x\subseteq Q$ denote the union 
of the boxes that are fully contained in the ball of radius $r_n$ and centered at $x$. 
Then, for any box $S\in \cS$ and $x\in S$ define 
$$A_x=\bigcap_{y\in S} A^\circ_y$$
(Figure~\ref{fig:discrete-step}).
Write $\cA_x$ for the collection of boxes whose union is $A_x$, and 
let $a$ denote the number of boxes that compose $\cA_x$, for $x\in [0,1]^2$. 
\nomenclature[A]{$\cA_x$}{The collection of boxes arranged in a ``discrete disk'' around 
the point $x$}
\nomenclature[a]{$a$}{The common cardinality of $\cA_x$, for $x\in [0,1]^2$}

\begin{figure}[t]
    \centering
    \includegraphics[scale=.9]{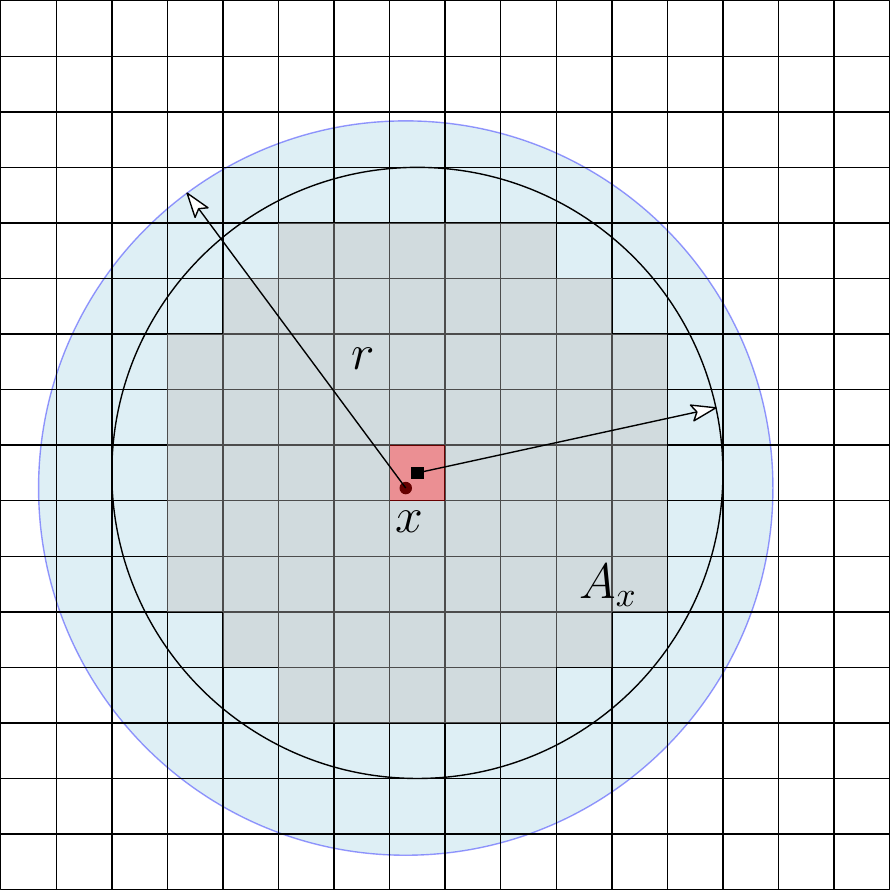}
    \caption{\label{fig:discrete-step}
    Any circle of radius $r_n$ centered at $x$ fully contains a $A_x$ which
    is a copy of a \emph{fixed} collection of boxes.}
\end{figure}

\begin{lemma}\label{lem:number-points}
For any $\delta\in (0,1/4)$ there exist constants $k_0$, $d_0$ and $n_0$ 
such that for all $k\ge k_0$, $d\ge d_0$ and $n\ge n_0$, we have
$$|a-\pi d^2|\le \delta d^2.$$
\end{lemma}
\begin{proof}
First, every box that is fully contained in $B(x,r_n)$ accounts for an area of 
$(r_n'/d)^2$, so we must have $a(r_n'/d)^2\le \pi r_n^2$. 
Now, since $(1-kr_n)r_n\le r_n'\le r_n$ (see just below the definition in \eqref{eq:def_m-rp}), 
it follows that
\begin{align*}
a 
\le \pi d^2 \pran{\frac{r_n}{r_n'}}^2
\le \pi d^2 \cdot \frac 1 {1-k r_n}
\le \pi d^2 (1 + \delta),
\end{align*}
provided that $kr_n\le \delta/2<1/4$.

On the other hand, the boxes that intersect the ball $B(x,r_n)$ but are not fully 
contained in it must touch the boundary of $B(x,r_n)$. So any such box must
lie entirely inside the annulus $B(x,r_n+r_n'\sqrt 2/d) \setminus B(x, r_n-r_n'\sqrt 2/d)$. 
In particular, the number of these boxes is at most
\begin{align*}
\frac{\pi}{(r_n'/d)^2} 
\Bigg\{ \pran{r_n+\frac{r_n'\sqrt 2}{d}}^2 - \pran{r_n-\frac{r_n'\sqrt 2}{d}}^2 \Bigg\} 
\le 4\pi \sqrt 2 d \frac {r_n} {r_n'}.
\end{align*}
Since $(1-kr_n)r_n\le r_n'\le r_n$, it follows that 
\begin{align*}
a
&\ge \pi d^2 \pran{\frac {r_n} {r_n'}}^2 - 4\pi d \sqrt 2 \frac {r_n} {r_n'}\\
&\ge \pi d^2 - 8 \pi d \sqrt 2,
\end{align*}
for all $k$ and $r_n$ such that $2kr_n\le 1$. The result follows readily. 
\end{proof}

\medskip
\noi\textsc{extracting a discrete branching random walk.}\ 
Next we obtain lower bounds for the sizes of neighbourhoods of points $x$ 
in the irrigation graph. 
It is here that the discretization in boxes is important since it 
ensures that the process $S(y)$, $y\in \Delta_x$, of boxes containing the 
points which may be reached from $x$ using directed edges \emph{dominates} 
a branching random walk. In order to properly extract this branching random walk 
on the set of boxes 
\begin{itemize}
    \item we artificially reduce the number of points in each box 
so that the distributions of the number of offspring of the spatial increments are 
\emph{homogeneous}, which fixes the spatial component, and 
    \item we then ensure that the offspring of every individual has the same distribution,
    so that the underlying genealogy is a Galton--Watson tree.
\end{itemize}

We now define the discrete branching random walk as a process indexed by the 
infinite plane tree $$\cU=\bigcup_{n\ge 0} \{1, 2, 3, \dots \}^n,$$ where the 
individuals in the $n$-th generation are represented by a word of length $n$ on the 
alphabet $\{1,2,\dots\}$. The tree $\cU$ is seen as rooted at the empty word $\varnothing$. 
The descendants of an individual $u$ are represented by the words have $u$ as a prefix. 
The children of $u\in \cU$ are $ui$, $i\ge 1$. If $v=ui$ for some $u\in \cU$ and $i\ge 1$,
$u$ is the parent of $v$ and is denoted by $p(v)=u$. 
\nomenclature[p]{$p(u)$}{The parent of $u\in \cU$, $u\ne \varnothing$}
For $u,v\in \cU$ we write $u\preceq v$ if $u$ is an ancestor of $v$, potentially $u=v$.

\medskip
\noi\textsc{exploring the neighborhoods in $\Gamma_n^+$ and the spatial component.}\ 
Consider any $\delta$-good cell $Q$, and a point $x\in Q$. 
We define the tree-indexed process $(Z^\circ_u, u\in \cU)$ corresponding to 
\nomenclature[Zo]{$Z^\circ$}{An auxiliary process used to define the branching random
walk $Z$}
the exploration of the neighborhoods of $x$ in the directed irrigation graphs $\Gamma_n^+$. 
To properly prune the tree we introduce a cemetery state $\partial$, that is assigned 
to the words of $\cU$ that do not correspond to a vertex of $\bX$. 
More precisely, we first define an auxiliary process $(Z^\circ_u, u\in \cU)$. 
We set $Z^\circ_\varnothing = x$; 
then if $Z^\circ_u=z\in \bX$ for some $u\in \cU$, 
we define $Z^\circ_{ui}=Y_i(z)$, for $1\le i\le \xi_z$, 
and $Z^\circ_{ui}=\partial$ for $i>\xi_z$ and similarly for any word $w$ with $ui$ as a prefix. 
Then, for $m\ge 0$ integer, $\{Z^\circ_u: |u|=m, Z^\circ_u\ne \partial\}$ 
is precisely the set of points of $\bX$ that can be reached 
from $x$ using a path of length exactly $m$. However, some points may appear more than
once in the process. 

We now \emph{skim} the process in order to extract a subprocess for which all the points 
are distinct, and for which we can guarantee that the process induced on the set of boxes
visited by the points is a branching random walk. Although it is not crucial, 
it is natural to skim the tree in the breadth-first order, where for $u,v\in \cU$ we write
$u\preceq_b v$ if $|u|<|v|$ or if there exists $w\in \cU$ and $u=wi$, $v=wj$ for $i<j$. 
We let $T_x(u)$ be the event that $Z^\circ_u\neq Z^\circ_v$ for every $v\preceq_b u$. 

The skimming is done by maintaining a set of \emph{valid} points, which 
ensures that a point choosen at random among the valid points in a 
certain subset of the boxes is contained in a uniformly random box in that subset.
Initially, we have a set of \emph{valid} points $V\subseteq \bX \setminus F$, that are points 
whose choices have not yet been exposed, but maybe not all such points. 
We choose $V$ in such a way that for every box $S\in \cS$, the number of 
elements of $V\cap S$ is the same, and we denote by $c$ this the common cardinality. 
We do this in such a way that  the set $V$ has maximal cardinality. 
Let $(w_i, i\ge 0)$ be the breadth-first 
ordering of the elements of the set $\{u:Z^\circ_{u}\neq \partial, |u|\le k^2\}$. 
Set $V_\varnothing=V$. 
If $x\not\in V_\varnothing$, set $Z^\di_\varnothing=\partial$ as well as for all the 
words $u\in \cU$; set $V_{w_1}=V_\varnothing$. 
\nomenclature[Zs]{$Z^\di$}{A auxiliary process used to define the branching random walk $Z$}
Otherwise $x\in V_\varnothing$, and we set $Z^\diamond_\varnothing = x$. 
Then, we update the set of valid points. 
For each box $S\in \cS\setminus \{S(x)\}$, let $X_{(\varnothing)}(S)$ be the 
point of $V_\varnothing \cap S$ which has minimum index in $\bX$, if such a point exists, 
or $X_{(\varnothing)}(S)=\partial$ otherwise. Then we set
$$V_{w_1}:=V_\varnothing \setminus \big( \{X_{(\varnothing)}(S): S\in \cS\} \cup \{x\}\big),$$
which ensures that the number of points in $V_{u_1}\cap S$ is the same and equal to $c-1$ 
for every box $S\in \cS$, since precisely one point has been removed from every box.

Suppose now that we have 
defined $Z^\diamond_{w_j}$ for all $j<i$, and $V_{w_j}$ for $j\le i$. 
Then, if $Z^\circ_{w_i}\in V_{w_i}\cap A_{Z_{p(w_i)}}$, we set
$Z^\diamond_{w_i}=Z^\circ_{w_i}$. Otherwise define $Z^\diamond_{w_i}=\partial$. Then, 
for every box $S\in \cS$, let $X_{(i)}(S)$ be the point in $V_{w_i}\cap S$ which 
has minimum index in $\bX$, and define
$$
V_{w_{i+1}}=
\left\{
\begin{array}{l l}
V_{w_i}\setminus \big(\{X_{(i)}(S): S\in \cS\setminus \{S(Z^\diamond_{w_i})\}\}
\cup \{Z^\diamond_{w_i}\} \big) & \text{if~}Z^\diamond_{w_i} \ne \partial\\
V_{w_i}\setminus \{X_{(i)}(S): S\in \cS\} & \text{if~}Z^\diamond_{w_i}=\partial.
\end{array}
\right.
$$

\medskip
\noi\textsc{skimming the underlying genealogy.}\ 
The process of interest is $(S(Z^\diamond_u), u\in \cU)$. Note that for $u,v\in \cU$, 
with $u$ the parent of $v$ in $\cU$, $u=p(v)$, conditional on 
$Z^\diamond_u,Z^\diamond_v\neq \partial$, 
and say $S(Z^\diamond_u)=s$, the box $S(Z^\diamond_v)$ which contains $Z_v$ is uniformly 
random in $\cA_s$. 
So for every sequence of words $(v_i, i\ge 0)$ in $\cU$ with $|v_i|=i$, 
conditional on $Z_{v_\ell}\ne \partial$, the process $(S(Z^\diamond_{v_i}))_{0\le i\le \ell}$ 
is a random walk with i.i.d.\ increments.
The only reason why the entire process $(S(Z^\diamond_u), u\in \cU)$
is not a branching random walk is that the individuals do not all jump to $\partial$
with the same probability (in other words the individuals 
do not all have the same offspring distribution) either because of the 
inhomogeneity of the point set $\bX$, or because of the changing number of valid points. 
We now construct the (truncated) branching random walk $(Z_u, |u|\le k^2)$ by 
homogenizing the offspring distribution. 
For $u\in \cU$, and $s\ge 0$, define 
\begin{equation}\label{eq:keep_child}
\alpha_i:=\frac{a(c-i)}{\rho(Z^\di_{p(w_i)})},
\end{equation}
that is the probability that the node $w_i$, which is a child of $p(w_i)$, is 
such that $Z^\di_{w_i}\ne \partial$. Let also $\alpha:=\inf\{\alpha_i: 1\le i\le i_m\}$,
where $i_m:=\#\{u:Z^\circ_{u}\neq \partial, |u|\le k^2\}$. 
Let $U_i$, $i\ge 1$, be a collection of i.i.d.\ random variables uniformly distributed on $[0,1]$, 
and finally define 
\[
Z_{w_i} = 
\left\{
\begin{array}{ll} 
Z^\diamond_{w_i} & \text{if~} U_i \le (1-\alpha)/(1-\alpha_i)\\
\partial & \text{otherwise}. 
\end{array} 
\right.
\]
\nomenclature[Z]{$Z$}{The branching random walk on $[m]^2$}
Then, for every $u$, $|u|<k$, writing $\zeta_u:=\#\{ui: Z_{ui}\ne \partial\}$ for the 
offspring of $u$, $(\zeta_u: |u|<k^2)$ is a collection of i.i.d.\ random variables; 
write $\zeta$ for the typical copy of this random variable.
\nomenclature[z]{$\zeta$}{The offspring distribution of the branching random walk $Z$}
More precisely, $\zeta_u$ is distributed like a binomial random variable with parameters 
$\xi_u$ and $\alpha$. In particular, 
\begin{equation}\label{eq:bound_alpha}
\alpha \ge \frac{(\pi-\delta)d^2 (\eta-\delta) \log n - \kappa^{2k^2}}
{(\pi+\delta) d^2 (\eta+\delta) \log n},
\end{equation}
and if we write $\Ec{\xi}=1+\epsilon$ for $\epsilon>0$, it is possible to choose $\delta$, 
$d_1$, $n_1$ large enough such that for $d\ge d_1$ and $n\ge n_1$, we have 
\begin{equation*}
\Ec{\zeta} \ge 1+\epsilon/2.
\end{equation*}

\subsection{Analyzing the discrete branching random walk.}

In this section, we slightly abuse notation and identify the set of boxes and their 
representation as the discrete torus. Furthermore, since for $n$ large enough, the 
difference between the torus and $Z^2$ cannot be felt by a walk of $k^2$ steps, we 
talk about $\Z^2$. In particular, we let $\cA$ denote the subset of $\Z^2$ corresponding 
to the boxes in $\cA_0$, which is the set of potential spatial increments of our walks.

We now consider the (truncated) \emph{branching random walk} $(Z_u,|u|\le k^2)$ taking values 
in $\Z^2$, that we complete into a branching random walk by generating the missing 
individuals using an independent family of random variables for the offspring and the 
spatial displacements. 
By definition, an individual $u$ located at $Z_u$ gives birth to $\zeta_u$
individuals, such that the displacements are i.i.d.\ uniform in $\cA$. 
Furthermore, every individual behaves in the same way and independently of the others. 
For $S\in \cS$ and $i\ge 0$, define 
$$M_i(S):=\#\{u\in \cU: |u|=i, Z_u\in S\},$$
the number of individuals $u\in \cU$ in generation $i$ such that $Z_u\in S$. 


\begin{lemma}\label{lem:fill-up}
Let $q>0$ be the extinction probability of the Galton--Watson process 
underlying the branching random walk $(Z_u)_{u\in \cU}$. Then, 
for all $k$ large enough, we have
$$\p{\#\{v\in \cU: |v|=k^2, Z_v\in S\} \le \Ec{\zeta}^{2k^2/3} } \le 2q.$$
\end{lemma}



Before proving Lemma~\ref{lem:fill-up}, we show that the extinction probability 
$q$ in the bound may be made as small as we want by choice of the constants. 
By the bound in \eqref{eq:bound_alpha}, this reduces to showing that the 
extinction probablity goes to zero as $\alpha\to 1$. 

\begin{lemma}\label{lem:bound_extinct}
Let $q$ be the extinction probability of a Galton--Watson process with 
offspring distribution $\zeta=\bin(\xi,\alpha)$ such that $\Ec{\xi}\alpha>1$
and $\xi\ge 1$. Then,
$$q\le \frac{1-\alpha}{1-\Ec{(1-\alpha)^\xi} - \Ec{\xi \alpha (1-\alpha)^{\xi-1}}}.$$
\end{lemma}
\begin{proof}To prove this, we use the standard fact that $q$ is the smallest $x\in[0,1]$ 
such that $x=\Ec{x^\zeta}$ \cite{AtNe1972}. Note the simple fact that if $f(x)$ and $g(x)$ 
are probability generating functions, then if $f(x)\le g(x)$ for all $x\in [0,1]$ the 
corresponding extinction probabilities $q_f$ and $q_g$ satisfy $q_f \le q_g$. So 
it suffices to find an upper bound on $\Ec{x^\zeta}$ which gives us a computable 
(and small) extinction probability. Writing $p_i=\p{\zeta=i}$, and $p_{\ge 2}=1-p_0-p_1$,
we have, for every $x\in [0,1]$,
\begin{align*}
\Ec{x^\zeta} 
& \le p_0+ (1-p_0-p_{\ge 2}) x + p_{\ge 2} x^2~.
\end{align*}
It follows readily that 
$$
q \le \frac{(p_0+p_{\ge 2})- |p_0-p_{\ge 2}|}{2 p_{\ge 2}}
=\frac{\min\{p_0,p_{\ge 2}\}}{p_{\ge 2}} \le \frac{p_0}{p_{\ge 2}}.
$$
Here, $p_{\ge 2}=1-\Ec{(1-\alpha)^\xi} - \Ec{\xi \alpha (1-\alpha)^{\xi-1}}$ and 
since $\xi\ge 1$, we have $p_0\le 1-\alpha$, which completes the proof.
\end{proof}

The proof of Lemma~\ref{lem:fill-up} goes in two steps. First, one shows that 
for some $\delta>0$, the branching random walk has at least 
$(1+\epsilon/2)^{\delta k}$ individuals in the $\delta k$-th generation, that 
all lie within distance $k/4$ of the center of the cell. We call such individuals 
\emph{decent}. The decent individuals are the starting points of independent branching 
random walks. In order to prove the claim, we show that, with probability no smaller 
than a polynomial in $1/k$, a single of these decent individuals produces 
enough decendants for $\#\{v\in \cU: |v|=k^2, Z_u\in S\}\ge \Ec{\zeta}^{2k^2/3}$ to occur. 

\begin{proof}[Proof of Lemma~\ref{lem:fill-up}]
Consider the genealogical tree of the branching random walk $(Z_u)_{u\in \cU}$, 
and write $(M_i)_{i\ge 0}$, for the associated Galton--Watson process. So we have
\nomenclature{M}{$(M_i)_{i\ge 0}$ the Galton--Watson process underlying the branching random 
walk $Z$}
\[
M_i = \#\{Z_u: u\in \cU, |u|=i\}.
\]
As we already mentioned, we have $\Ec{\zeta}>1$ and the process is supercritical. 
Furthermore, the offspring distribution is bounded ($\zeta\le \kappa$) 
so that Doob's limit law \cite{AtNe1972} implies that, as $m\to\infty$, we have
as $\ell\to\infty$,
$$\frac{M_\ell}{\Ec{M_1}^\ell} \to W$$
in distribution, for some random variable $W$ that is absolutely 
continuous, except possibly at $0$. Furthermore, the limit random 
variable satisfies $\p{W=0}=q$, where $q$ is the extinction probability 
of the Galton--Watson process $(M_i)_{i\ge 0}$. 

Since $q\in (0,1)$, we can find a $\beta>0$ such that $\Cprob{2\beta <W<1/(2\beta)}{W>0}>1-q$. 
It follows that
\[\liminf_{\ell\to\infty} \p{\frac{M_\ell}{\Ec{M_1}^\ell}\in \left[\beta, \frac 1\beta\right]}
\ge \p{W\in \left[2\beta,\frac 1{2\beta}\right]} 
> (1-q)^2,
\]
and in particular, for any $\mu\in (0,1/2)$ and $k$ large enough, 
\[
\p{M_{\lfloor \mu k\rfloor } \ge \beta \Ec{M_1}^{\mu k -1}} \ge (1-q)^2
\]

Recall that an individual $v$ is \emph{decent} if $\|Z_v\|\le k/4$, where $\|\cdot\|$ 
denotes the Euclidean distance. However, the spatial increments are bounded, and for every $v$ such 
that $|v|=\lfloor \mu k\rfloor$, we have 
$$\|Z_v\|\le \mu \le \mu k 2 d.$$
It follows that for $\mu\in (0,1/(8d))$, \emph{every} individual $v$ 
with $|v|=\lfloor \mu k \rfloor$ is decent.
Fix now such a $\mu$. Writing $D_m$ for the number of decent individuals at level $m$, we have 
\begin{equation}\label{eq:num_decent}
\p{D_{\lfloor \mu k \rfloor} < \beta \Ec{M_1}^{\mu k-1}} \le 1-(1-q)^2.
\end{equation}

For every decent individual at depth $\lfloor \mu k\rfloor$, there is a subtree 
that might well give us enough individuals at generation $k^2$ all lying 
in $S$. In order to ensure some level of concentration, we only consider 
the individuals $u$, $|u|=\lfloor \mu k\rfloor$, for which the corresponding 
Doob limit $W_u$ in the subtree rooted at $u$ satisfies $2\beta<W_u<1/(2\beta)$. 
For $\ell\ge 0$ and $u$ such that $|u|\le \ell$ write 
$$M_\ell(u):=\#\{v: u\preceq v, |v|=\ell\}.$$
Then, for all $k$ large enough, 
\begin{align*}
\CExp{M_{k^2}(u)}{2\beta<W_u<1/(2\beta), u~\text{decent}}
& \ge \beta \Ec{M_1}^{k^2-\lfloor \mu k\rfloor} \cdot k^{-c}
\end{align*}
for some $c>0$ whose existence is guaranteed by Lemma~\ref{lem:fill-up-one} below. 
However, for every such individual $u$, 
on the event $\{2\beta<W_u<1/(2\beta)\}$, we have for all $k$ large enough
$$
M_{k^2}(u,S)
:=\#\big\{v:u\preceq v, |v|=k^2, Z_v\in S\big\}
\le \beta^{-1}\Ec{M_1}^{k^2-\lfloor \mu k\rfloor}.$$
It follows that 
$$
\Cprob{M_{k^2}(u,S)\ge \frac \beta 2 
\Ec{M_1}^{k^2-\lfloor \mu k \rfloor} k^{-c}}{2\beta <W_u<1/(2\beta), u~\text{decent}} 
\ge \frac \beta 2 \cdot k^{-c},
$$
hence 
\begin{equation}\label{eq:desc_decent}
\Cprob{M_{k^2}(u,S)\ge \frac \beta 2 
\Ec{M_1}^{k^2-\lfloor \mu k \rfloor} k^{-c}}{u~\text{decent}}
\ge \frac \beta 2 k^{-c} (1-q)^2.
\end{equation}

Finally, combining \eqref{eq:num_decent} and \eqref{eq:desc_decent}, we see that, 
for $k$ large enough, the probability that we do not have at least $\Ec{M_1}^{2k^2/3}$ 
individuals of the $k^2$-th generation that lie in $B$ is at most 
\begin{align*}
&\p{D_{\lfloor \mu k \rfloor} < \beta \Ec{M_1}^{\mu k-1}}
+ \Cprob{M_{k^2}(u,S)<\Ec{M_1}^{2k^2/3}}{u~\text{decent}}^{\beta \Ec{M_1}^{\mu k -1}}\\
&\le 1-(1-q)^2+\kappa^{-k\mu}
+\pran{1-\frac \beta 2 k^{-c}(1-q)^2}^{\beta \Ec{M_1}^{\mu k}}\\
&\le 2 q,
\end{align*}
for $k$ sufficiently large.
\end{proof}

It remains to prove the following key ingredient of the proof of 
Lemma \ref{lem:fill-up}.

\begin{lemma}\label{lem:fill-up-one}
Let $(R_i)_{i\ge 0}$ be a random walk on $\Z^2$ where the increments 
are i.i.d.\ uniformly random in $\cA$
(where $\cA$ is defined above just before Lemma \ref{lem:fill-up}). 
Then, for any $\mu\in (0,1/2)$, there exists a constant $c>0$,
such that for any 
$x\in \{-\lfloor kd/4\rfloor, \dots, \lfloor kd/4\rfloor\}^2$
and $y\in\{-\lfloor kd/2\rfloor, \dots, \lfloor kd/2\rfloor\}^2$,
$$\Cprob{R_{k^2-\mu k}=y; R_i\in \cS, 0\le i\le k^2-\mu k}
{R_0=x}\ge k^{-c},$$
for all $k$ large enough. 
\end{lemma}
\nomenclature[R]{$(R_i)_{i\ge 0}$}{The random walk on $\Z^2$ used in the strong coupling}
\begin{proof}
Let $\cS=[-\lfloor kd/2\rfloor, \lfloor kd/2\rfloor]^2$ be the scaled version of the cell. 
The argument relies on the strong embedding theorem of \citet*{KoMaTu1975a} or, 
more precisely, its multidimensional version by \citet{Zaitsev1998a} 
(see also \cite{Ei1989a}): 
there exists a coupling of $(R_i)_{0\le i\le k^2}$ with a Brownian motion 
$(\Xi_t)_{0\le t\le k^2}$ such that for every $c_2>0$ there exists a $c_1>0$ such that, 
for every $k$ large enough,
\begin{equation}\label{eq:kmt}
\p{\max_{0\le i\le k^2}\| R_i-\Xi_i\| \ge c_1 \log k } \le k^{-c_2},
\end{equation}
where $\|\cdot\|$ denotes Euclidean norm in $\R^2$.  
We now consider such a coupling, and for a constant $c_1$ to be chosen later,
let $E=E(c_1)$ be the event that $\|R_i-\Xi_i\|\le c_1 \log k$
for every $0\le i\le k^2$. 
Let $y\in \{-\lfloor kd/2\rfloor, \dots, \lfloor kd/2\rfloor\}^2$, and 
recall that $S(y)$ denotes the corresponding box in $\cS$. 
On $E$, if it turns out that $\Xi_{k^2-\mu k-m}\in S(y)$, 
then $R_{k^2-\mu k -m}$ is reasonably close to $y$ and
there is a decent chance that it ends up at $y$ at time $k^2-\mu k$. 
More precisely, if for some integer $m\le k$ 
we have $\Xi_{k^2-\mu k-m}\in S(y)$, then $\|R_{k^2-\mu k -m}-y\|\le c_1 \log k$, and 
we let $H=H(c_1,m)$ denote the latter event. Then, we have 
\begin{align*}
& \p{H, R_i\in \cS, 0\le i\le k^2-\mu k-m} \\
& \ge \p{H, R_i\in \cS, 0\le i\le k^2-\mu k-m, E} \\
& \ge \p{\Xi_{k^2-\mu k-m}=S(y), \Xi_i\in \cS, 0\le i\le k^2-\mu k-m, E}\\
& \ge \p{\Xi_{k^2-\mu k-m}=S(y), \Xi_i\in \cS, 0\le i\le k^2-\mu k-m} - \p{E^c}.
\end{align*}
Now, by the local limit theorem, for all $k$ large enough, one has 
$$
\inf_{0\le m\le k}
\p{\Xi_{k^2-\mu k-m} \in S(y) ;\inf_{1\le i\le k^2} d(\Xi_i,\cS^c)\le c_1 \log k} 
\ge k^{-2}
$$
where $d(x,\cS^c)$ denotes the distance from $x\in \R^2$ to the set $\cS^c$. 
Choosing $c_1$ be the constant such that $c_2=3$ in \eqref{eq:kmt}, we obtain 
\begin{equation}\label{eq:rw_core}
\inf_{0\le m\le k}\p{H, R_i\in \cS, 0\le i\le k^2-\mu k-m} \ge k^{-3},
\end{equation}
for all $k$ large enough.
In particular, with $m=\lfloor c_1 \log k/(2d)\rfloor$ it is possible for the random walk
to go to $y$ within the $m$ steps, while staying within $\cS$. 
It follows that, with $a:=|\cA|$ the number of potential increments at every step,
\begin{equation}\label{eq:rw_laststeps}
\Cprob{R_{k^2-\mu k}=y,R_{k^2-\mu k -i }\in \cS, 0\le i\le m}{H(c_1,m)}
\ge a^{-m}.
\end{equation}
Putting \eqref{eq:rw_core} and \eqref{eq:rw_laststeps} together completes the proof 
for $c=3+c_1 \log a$.
\end{proof}

\section{An upper bound on the size of the largest component for $c=1$}
\label{sec:c1}

In this section, we prove Theorem~\ref{thm:sublinear_critical} about the size 
of the largest component of $\Gamma_n(r_n,1)$. Write $\sC_1=\sC_1(\Gamma_n(r_n,1))$ 
for the size of the largest connected component. 
Although Theorem~\ref{thm:sublinear_critical} is suboptimal, the condition on $r_n$ cannot 
be replaced altogether, because it is easy to show that for fixed $r_n > 0 $, 
$\sC_1= \Theta (n)$ with high probability when $\xi=1$ almost surely. 
This is also the case for sequences $r_n$ that tend to $0$ slowly with $n$. 

The main technical result is the following tail bound on the size of the 
largest connected component. 

\begin{lemma}\label{lem:upper_c1}
Let $r_n > 0, t \ge 1, \epsilon > 0$. Then, 
$$
\p{ \sC_1 \ge 2 + (1+tnr_n^2)^3 (1+\epsilon)^2 \log^2 n }
\le n^{-\epsilon + \frac 1 {1 + tnr_n^2} } + n^2 e^{(n-2)\pi r_n^2 ( t - 1 - t \log t)} .
$$
\end{lemma}
\begin{proof}
For $\xi=1$, the structure of the graph is that of a mapping and $\Gamma_n$ is
of a collection of connected components each of which consist of
either a tree
or a unique cycle 
from which some trees are pending. In order to bound the size $\sC_1$ of the largest 
connected component, we first bound the length of the longest directed path in $\Gamma_n^+$. 
Since the edges bind vertices that are at most $r_n$ apart, this bounds the extent of the 
connected components hence their sizes. 

Recall that $\rho_{r_n}(x)=|B(x,r_n)\cap \bX|$ denotes the number of $X_i$'s in $B(x,r_n)$.
We first show that for $t > 1$ (and thus, $t-1-t \log t < 0$), we have
$$
\p{ \max_{1\le i\le n} \sup_{s \ge r_n} \frac {\rho_s(X_i) -2}{ns^2} \ge t }
\le
n^2 e^{(n-2)\pi r_n^2 ( t - 1 - t \log t)}.
$$
Observe that the supremum in this inequality is reached for $\rho_s(X_i)$
for some $s=\|X_i-X_j\|$. Also, $\rho_s(x)$ is distributed as a binomial 
random variable with parameters $n$ and $\pi s^2$ (we are in the torus), 
and $\rho_{\|X_i-X_j\|}(X_i)$ is approximately equal to $2 + \bin(n-2; \pi \|X_i-X_j\|^2)$. 
By Chernoff's bound \cite{Che52} (see also \cite{DeZe1998,
JaLuRu2000}), for $u > 1$,
$$
\p{\bin(k;p)\ge u k p} \le e^{kp(u-1-u\log u)},
$$
so that here, we have
$$
\p{\bin(n-2, \pi s^2) \ge 2 + u (n-2) \pi s^2 } 
\le e^{(n-2)\pi s^2 (u- 1 - u \log u)}.
$$
Thus,
\begin{align*}
\p{ \max_{1\le i\le n} \sup_{s \ge r_n} \frac {\rho_s(X_i) -2}{ns^2} \ge t }
&\le \binom n 2 \sup_{s \ge r_n} e^{(n-2)\pi s^2 ( t - 1 - t \log t ) } \\
&\le n^2 e^{(n-2)\pi r_n^2 ( t - 1 - t \log t)}.
\end{align*}
Introduce the event 
$$
A:=\left\{\max_{1\le i\le n} \sup_{s \ge r_n} \frac {\rho_s(X_i) -2}{ns^2} < t\right\}.
$$
Starting from a vertex $i$, we can follow the directed links in $\Gamma_n^+$,
forming a \emph{maximal} path $P_i$ of distinct vertices. 
The last vertex $j$ in this path must be pointing towards a vertex $k$ of $P_i$ 
(potentially itself). 
From each vertex in $P_i$ the probability of linking to a $k$ higher up in the path 
is at least 
$$
\frac 1 {\rho_{r_n}(X_i) -1 } \ge \frac 1 {1 + t n r_n^2 }
$$
if $A$ occurs. Writing $|P_i|$ for the number of vertices of $P_i$, we see that,
since the choices of links are independent,
$$
\p{|P_i| > \ell } \le \left( 1- \frac 1 {1 + t n r_n^2 } \right)^\ell.
$$
By the union bound, conditional on $X_1,\ldots,X_n$ such that $A$ holds,
$$
\p{ \max_{1\le i\le n} |P_i| > \ell } 
\le n \left( 1 - \frac 1 {1 + t n r_n^2 } \right)^\ell 
\le n \exp \left(- \frac \ell {1 + t nr_n^2} \right).
$$

Now, if the maximum length of a directed path $\max_i |P_i|$ is no more than $\ell$, 
then every vertex is within $\ell$ edges of \emph{any} vertex of the unique cycle of 
the connected component. Thus, if this occurs, then every connected component 
is contained within a ball $B(X_j,r_n\ell)$ for some $1\le j\le n$. It follows that 
\begin{align*}
\p { \sC_1 \ge 2 + n (r_n \ell)^2 t }
&\le \p{ A^c } + \p{ A, \sC_1 \ge 2 + n (r_n \ell)^2 t } \\
&\le \p{ A^c } + \p{ A, \max_{1\le i\le n} |P_i| > \ell } \\
&\le n^2 e^{(n-2)\pi r_n^2 ( t - 1 - t \log t)} 
+ n e^{- \frac \ell {1 + t nr_n^2}} . 
\end{align*}
For fixed $\epsilon > 0$, take $\ell = \lfloor { (1 + tnr_n^2) (1+\epsilon) \log n } \rfloor.$
We conclude that
$$
\p{ \sC_1 \ge 2 + (1+tnr_n^2)^3 (1+\epsilon)^2 \log^2 n }
\le n^{-\epsilon + \frac 1 {1+tnr_n^2}} + n^2 e^{(n-2)\pi r_n^2 ( t - 1 - t \log t )},
$$
which completes the proof of the lemma.
\end{proof}

\begin{proof}[Proof of Theorem~\ref{thm:sublinear_critical}]
Lemma~\ref{lem:upper_c1} can be used for various ranges of $r_n$. In the entire proof, we 
use it with $\epsilon = 2$ to ensure that the first term in the upper bound there is 
$o(1)$.
We split the region $r_n\in [0, o(n \log n)^{-1/3}]$ into two, and first consider
$$
r_n \le \sqrt{ \frac{\log n}{\pi n}  }.
$$
In this range, we define $t$ as the solution of
$$
t \log t + 1 - t = \frac { 3 \log n}{ \pi n r_n^2 }.
$$
Observe that since the right-hand side is at least $3>1$, there is indeed a unique solution. 
Note that this solution could have an infinite limit supremum, but
its limit infimum is larger than one (so that one can use Lemma~\ref{lem:upper_c1} with this 
value for $t$). Moreover, one has
$$
t = \Theta \left( \frac { 3 \log n} {\pi n r_n^2 \log 
\big( \frac { 3 \log n}{\pi n r_n^2 } \big) } \right),
$$
so that
$$
(1+tnr_n^2)^3 (\log n)^2 
= \Theta \left( \frac { \log^5 n}{\log^3 
\big( \frac { 3 \log n}{\pi n r_n^2 } \big) } \right)
\le \Theta ( \log^5 n ).
$$
By Lemma~\ref{lem:upper_c1}, in this range of $r_n$, we have $\sC_1 \le C \log^5 n$ with 
probability tending to one as $n \to \infty$, where $C$ is a fixed constant (uniform over all 
sequences $r_n$ in this range).

Next, consider 
$$
r_n \ge \sqrt{ \frac {\log n}{\pi n}  }.
$$
Define $t_0 = 3.59112167\ldots$ as the unique solution greater than one of
$t_0 \log t_0 = t_0 + 1$. With this choice, if $t > t_0$, the upper bound 
in the inequality of Lemma~\ref{lem:upper_c1} is $o(1)$.
Note that
$$
(1+tnr_n^2)^3 (\log n)^2 
= \Theta \left( n^3 (\log n)^2 r_n^6 \right),
$$
which is $o(n)$ if $r_n = o( (n \log n)^{-1/3} )$. Overall, we have proved that 
$\sC_1=o(n)$ as long as $r_n=o((n\log n)^{-1/3})$. 
\end{proof}

\section{Concluding remarks and open questions}

From a practical point of view, the sparsification done via irrigation graphs
is especially interesting since an average degree of $(1+\epsilon)$
guarantees that the majority of the nodes are part of the network. It is proved 
in \cite{BrDeFrLu2011a} that catching all the outsiders would require an 
average degree of about $\Theta(\sqrt{\log n/\log\log n})$, so that it might not 
be worth the cost. 

Theorem~\ref{thm:sublinear_critical} is suboptimal in the range it allows for $r$, and 
it would be interesting to find a wider range of $r$ for which 
one does not have a connected component of linear size. It is not quite clear 
that there is a threshold since the property that there exists 
a connected component of size at least $c n$ is not clearly monotonic in $r$ for 
fixed $\xi$. It would be of interest to know whether the property that 
a giant exists with high probability is monotonic in $r_n$ (for fixed $\xi$):
is it the case that if a giant exists whp for a given $r_n$ and fixed $\xi$,
then a giant also exists whp for any sequence $r_n'$ with $r_n' \ge r_n$ and the
same fixed $\xi$? Assuming this is the case, it would be interesting to study where 
the threshold $r^\star=r^\star(\xi)$ is for the existence of a giant when $\xi=1$, 
but also for other (constant) values.

The question of the spanning ratio of the giant component is another interesting one. 
Of course, for $\xi$ such that $\e \xi\ge 1+\epsilon$, the largest connected 
component has unbounded spanning ratio if we consider the definition
$$
\max_{1\le i,j\le n} \frac{\|X_i-X_j\|}{d_{\Gamma}(i,j)},
$$
where $d_{\Gamma}$ denotes the graph distance in $\Gamma_n(r_n,\xi)$. 
However, even if we disallow the pairs of points that are
either disconnected or
 too close, that is 
for which $\|X_i-X_j\|\le r$, it is not clear that the ratio becomes bounded. 
Indeed, our construction only guarantees that most points in the same cells get 
connected via two webs that hook up potentially far from that cell. 
In \cite{BrDeFrLu2011a} it is shown that the spanning ratio of $\Gamma(r_n,c_n)$
is bounded whp when $r_n\ge \gamma \sqrt{\log n/n}$ and $c_n \ge \mu\sqrt{\log n}$ 
for sufficiently large constants $\gamma$ and $\mu$.

Finally, our techniques only show that when $\E{\xi}>1$ the largest connected 
component spans most of the vertices, but we have no control on the number of 
vertices that are left over. The question of the size of the second largest 
connected component may possibly be tackled by guessing which configurations 
are most ``economical'' in terms of avoiding to connect to the outside world, 
as in \cite{BrDeFrLu2011a}.

\section*{Acknowledgement}
Part of this work has been conducted while attending a workshop at the
Banff International Research Station for the workshop on Models of
Sparse Graphs and Network Algorithms.

\appendix

\section{Proof of uniformity lemma}

\begin{proof}[Proof of Lemma~\ref{lem:strong_uniform}]
For any box $S$, the number points $|\bX \cap S|$ is distributed like a binomial 
random variable with parameters $n$ and $r'^2/(4d^2)$. 
By a classical concentration bound for binomial random variables 
\cite[see, e.g.,][]{JaLuRu2000,BoLuMa2012a}, we have for $\delta\in (0,1)$ and $p\in (0,1)$,
\begin{equation}\label{eq:tail_binom}
\pc{|\bin(n,p)-np|\ge \delta np} \le 2 e^{-np \delta^2/3 }.
\end{equation}
Now, every cell $Q$ contains $k^2d^2$ boxes, and by the union bound we have, 
for all $n$ large enough,
\begin{align*}
\p{Q\text{ is not } \delta\text{-good}} 
& \le 2 k^2d^2 e^{-nr_n'^2 \delta^2/(3\cdot 4d^2)}\\
& \le 2 k^2d^2 n^{-\gamma^2 \delta^2/(24 d^2)},
\end{align*}
since $\sqrt 2 r_n' \ge r_n$ for any $k\ge 1$ and all $n$ large enough. Furthermore, 
if there exists one cell that is not $\delta$-good, then one of the $(mkd)^2$ boxes
has a number of points that is out of range, so that as $n\to\infty$,
\begin{align*}
\p{\exists Q: Q \text{ is not } \delta\text{-good}} 
& \le 2 (mkd)^2 n^{-\gamma^2 \delta^2/(24 d^2)}\\
& \le n^{1-\gamma^2 \delta^2/(24 d^2) + o(1)},
\end{align*}
which tends to zero provided that $\gamma^2\ge 24 d^2/\delta^2.$
\end{proof}


\setlength{\bibsep}{.3em}
\bibliographystyle{plainnat}
\bibliography{connectivity}

\newpage
\small
\printnomenclature

\end{document}